\documentclass[a4paper,11pt]{amsart}
\usepackage{amsfonts}
\usepackage{amssymb}
\usepackage[utf8]{inputenc}
\usepackage{amsmath}
\usepackage{pdflscape}
\usepackage{easyReview}
\usepackage{graphicx}
\usepackage{tcolorbox}
\usepackage[colorlinks=true, linkcolor=red, citecolor=blue]{hyperref}
\usepackage[]{epsfig}
\usepackage{tikz}
\usetikzlibrary{decorations.markings}

\newtheorem{theorem}{Theorem}[section]
\newtheorem{proposition}[theorem]{Proposition}
\newtheorem{lemma}[theorem]{Lemma}
\newtheorem{corollary}[theorem]{Corollary}

\newtheorem{remark}{Remark}
\theoremstyle{remark}
\newtheorem*{remarks}{Remarks}

\usepackage{mathrsfs}

\numberwithin{equation}{section}
\makeatletter
\@namedef{subjclassname@2010}{\textup{2020} Mathematics Subject Classification}
\makeatother

\begin{document}
\title[Control of NLS in Exterior Domains]{Nonlinear Schrödinger equation in an exterior domain}
\author[Capistrano--Filho]{R. A. Capistrano--Filho*}
\address{Departamento de Matem\'atica, Universidade Federal de Pernambuco, Recife, PE 50740545,  Brazil}
\email{\url{roberto.capistranofilho@ufpe.br}}
\author[Rosier]{L. Rosier}
\address{Université du Littoral Côte d’Opale, Laboratoire de Mathématiques Pures et Appliquées J. Liouville, B.P. 699, F-62228 Calais, France}
\email{\url{lionel.rosier@univ-littoral.fr}}
\author[Zhang]{B.-Y. Zhang}
\address{Department of Mathematical Sciences,  University of Cincinnati, Cincinnati, Ohio 45221, USA}
\email{\url{zhangb@ucmail.uc.edu}}
\thanks{*Corresponding author: \url{roberto.capistranofilho@ufpe.br}}
\thanks{Capistrano--Filho was partially supported by the CAPES-COFECUB program grant number 88887.879175/2023-00, CNPq grant numbers 301744/2025-4 and 421573/2023-6, and PROPG (UFPE). Rosier is supported by ANR-24-CE40-5470 and ANR-25-CE23-6846-02.}
\subjclass[2010]{93B05, 93C20, 35Q55}
\date{\today}
\keywords{Nonlinear Schr\"{o}dinger equation, exact controllability, boundary controllability, exterior domains, local smoothing}
\begin{abstract}
We study the exact controllability of the cubic nonlinear Schr\"odinger equation in the exterior $\Omega=\mathbb{R}^n\setminus\Theta$ of a non-trapping obstacle $\Theta$, with internal controls supported outside a large ball, and the corresponding boundary control problem on $\Omega_0=B_{R_0}\setminus\Theta$ with Dirichlet controls acting only on the outer sphere $\partial B_{R_0}$. Using the local smoothing effect of Burq, G\'erard and Tzvetkov, we prove observability inequalities for the linear Schr\"odinger equation in the whole scale of Sobolev spaces $H^\sigma_D(\Omega)$, $\sigma\in[-2,2]$, associated with the Dirichlet Laplacian; they require only the non-trapping assumption (no star-shapedness of $\Theta$), and their proofs use no normal boundary traces. Combined with Strichartz-type estimates and a perturbation argument, this yields local exact controllability in $H^1_0(\Omega)$ and in $H^2(\Omega)\cap H^1_0(\Omega)$, in dimensions two and three, around the zero solution and, for internal controls and under a unique continuation assumption, around nontrivial trajectories. The same method applies to the quintic equation, which is energy-critical in dimension three. Several open problems are discussed.
\end{abstract}
\maketitle
\section{Introduction\label{Sec0}}

\subsection{Setting of the problem}
The cubic nonlinear Schr\"odinger equation
\begin{equation}
\label{s1}
iu_t+\Delta u +\lambda|u|^2u=0, \qquad (x,t)\in\Omega_0\times(0,T),
\end{equation}
where $\lambda\in\mathbb{R}\setminus\{0\}$, is one of the basic models for the propagation of nonlinear dispersive waves, with applications ranging from nonlinear optics to Bose--Einstein condensation. This paper is concerned with its controllability when the state lives in a bounded domain $\Omega_0\subset\mathbb{R}^n$ and the control acts through the boundary $\Gamma:=\partial\Omega_0$, either as a Dirichlet datum
\begin{equation}
u(x,t)=h(x,t),\qquad (x,t)\in\Gamma\times(0,T),
\label{s1_a}
\end{equation}
or as a Neumann datum
\begin{equation}
\frac{\partial u}{\partial\nu}(x,t)=h(x,t),\qquad (x,t)\in\Gamma\times(0,T).
\label{s2_a}
\end{equation}
The basic question is the following.

\vglue 0.2 cm
\noindent\textbf{Question $\mathcal{A}$ -- [Boundary control problem]:} \textit{Let $s\geq0$ and $T>0$ be given. For any $u_0,u_1\in H^s(\Omega_0)$, can one find a boundary control $h$ such that \eqref{s1}--\eqref{s1_a} (or \eqref{s1}--\eqref{s2_a}) admits a solution $u\in C([0,T];H^s(\Omega_0))$ satisfying $u(x,0)=u_0(x)$ and $u(x,T)=u_1(x)$?}
\vglue 0.2 cm

When the control acts on the \emph{whole} boundary, Rosier and Zhang \cite{RoZhaJDE} answered Question $\mathcal{A}$: the systems \eqref{s1}--\eqref{s1_a} and \eqref{s1}--\eqref{s2_a} are locally exactly controllable in $H^s(\Omega_0)$, for a range of $s$ depending on $n$, around any smooth solution. Controls acting on the whole boundary are, however, rarely available in practice, and this leads to the following question.

\vglue 0.2 cm
\noindent\textbf{Question $\mathcal{B}$:}\textit{ Does Question $\mathcal{A}$ still have a positive answer when the control $h$ acts only on a part of $\Gamma$?}
\vglue 0.2 cm

We address Question $\mathcal{B}$ in a geometry which is natural from the point of view of scattering theory: a domain surrounding an obstacle, controlled from its outer boundary. Let $\Theta\subset\mathbb{R}^n$, $n\geq2$, be a nonempty compact set with smooth boundary such that $\Omega:=\mathbb{R}^n\setminus\Theta$ is connected, and assume that $\Theta$ is \emph{non-trapping}\footnote{Any light ray reflecting on the boundary of $\Theta$ according to the laws of geometric optics leaves any compact set in finite time.}. Fix $R_0>0$ such that $\Theta\subset B_{R_0}(0)$ and set
\begin{equation}\label{cond_geo}
\Omega_0:=\Omega\cap B_{R_0}(0)=B_{R_0}(0)\setminus\Theta,\qquad \Gamma=\Gamma_1\cup\Gamma_2,\qquad \Gamma_1:=\partial B_{R_0}(0),\quad \Gamma_2:=\partial\Theta .
\end{equation}
The control acts on the outer boundary $\Gamma_1$, while the boundary $\Gamma_2$ of the obstacle is not controlled (see Figure \ref{fig_geo}(a)), that is,
\begin{equation}
\label{s2}
u(x,t)=\begin{cases}
h(x,t), & x\in\Gamma_1,\\ 0, & x\in\Gamma_2,
\end{cases}
\qquad\text{(Dirichlet)}
\end{equation}
or
\begin{equation}
\label{s3}
\frac{\partial u}{\partial\nu}(x,t)=\begin{cases}
h(x,t), & x\in\Gamma_1,\\ 0, & x\in\Gamma_2,
\end{cases}
\qquad\text{(Neumann)}.
\end{equation}
Question $\mathcal{B}$ then becomes:

\vglue 0.2 cm
\noindent\textbf{Question $\mathcal{B}'$ -- [Boundary control from the outer boundary]:} \textit{Let $s\geq0$ and $T>0$ be given. For any $u_0,u_1\in H^s(\Omega_0)$, can one find a control $h$ such that \eqref{s1}--\eqref{s2} (or \eqref{s1}--\eqref{s3}) admits a solution $u\in C([0,T];H^s(\Omega_0))$ satisfying $u(x,0)=u_0(x)$ and $u(x,T)=u_1(x)$?}
\vglue 0.2 cm

In this work we give a positive answer to Question $\mathcal{B}'$ for the Dirichlet problem \eqref{s1}--\eqref{s2} in a local form: for $n\in\{2,3\}$ and $s\in\{1,2\}$, for small data in $H^s(\Omega_0)$ vanishing on $\Gamma_2$, and for every $T>0$ (Theorem \ref{principal}). The Neumann problem is discussed in Section \ref{Sec5}.

\subsection{Previous works}
The controllability of the Schr\"odinger equation has received a lot of attention in the last decades; we refer to \cite{Zuazua} for a survey of the contributions up to 2003, to \cite{KoLo,Miller,Phung,RaTaTeTu} for control issues, and to \cite{BaPu,CaGa,CaCriGa,LaTriZhang,MeOsRo,YuYa} for Carleman estimates and their applications to inverse problems. For the \emph{linear} equation, boundary controllability from a part of the boundary is classical: Machtyngier \cite{Mach} obtained it under the multiplier condition of Remark \ref{rem_geo}, and Burq \cite{Burq93} (for the plate equation; see Burq and Zworski \cite{BuZw} for the Schr\"odinger equation) proved controllability from the outer boundary for certain configurations of disjoint strictly convex obstacles satisfying additional geometric assumptions (see \cite[Section 6.2]{BuZw}), despite the presence of trapped rays, with an arbitrarily small loss of derivatives: initial data in $H^{1+\varepsilon}_0$, for any $\varepsilon>0$, are steered to zero by Dirichlet controls in $L^2$ in space and time on the outer boundary.

For the \emph{nonlinear} equation, the picture is much less complete. Illner, Lange and Teismann \cite{IlLaTe,IlLaTe1} proved local exact internal controllability on an interval with periodic boundary conditions, in $H^1_p(-\pi,\pi):=\{v\in H^1(-\pi,\pi):v(-\pi)=v(\pi)\}$, and Lange and Teismann \cite{LaTe} obtained local controllability in $H^1_0$ around a ground state for Dirichlet boundary conditions. Rosier and Zhang \cite{RoZhaSIAM} extended these results to $H^s$ for every $s\geq0$ and treated boundary controls acting at one end of the interval; Laurent \cite{Laurent} proved semi-global controllability and stabilization in this one-dimensional setting. In higher dimensions, Rosier and Zhang \cite{RoZhaJDE} obtained local exact boundary controllability on bounded domains with controls acting on the whole boundary, and \cite{RoZhaMMM} treated rectangles, combining sharp linear results with Bourgain spaces. On compact manifolds without boundary, Dehman, G\'erard and Lebeau \cite{DeGeLe} (dimension two) and Laurent \cite{Laurent2} (dimension three) proved semi-global internal controllability and stabilization under the geometric control condition and a unique continuation assumption; see also \cite{CaPa} for the fractional equation.

Exterior domains have been studied mainly from the point of view of the Cauchy problem and of stabilization. Burq, G\'erard and Tzvetkov \cite{BGT1} proved the local smoothing effect and Strichartz estimates for non-trapping obstacles and deduced global well-posedness results. Strichartz estimates with loss on general bounded or exterior domains of $\mathbb{R}^2$ and $\mathbb{R}^3$, without any non-trapping assumption, were obtained by Anton \cite{Anton}, and scale-invariant Strichartz estimates in a restricted range of exponents, in the exterior of non-trapping obstacles, by Blair, Smith and Sogge \cite{BSS}. In dimension three, the energy-critical quintic equation is well posed in $H^1_0(\Omega)$ for non-trapping obstacles, locally in time and globally for small data, by Ivanovici and Planchon \cite{IP}; see also \cite[Section 5]{BSS} for a simpler proof. For the Schr\"odinger equation with a localized damping, Aloui and Khenissi \cite{AlKh} proved uniform local energy decay estimates under an exterior geometric control condition, and Bortot and Cavalcanti \cite{BoCa} proved the exponential decay of the $L^2$ norm in the exterior of any non-trapping obstacle, for a damping (possibly nonlinear) effective outside a ball; their proof combines the local smoothing effect of \cite{BGT1} with a compactness--uniqueness argument. Cavalcanti, Corr\^ea and Domingos Cavalcanti \cite{CaCoCa} proved exponential stabilization of a damped defocusing NLS in two-dimensional exterior domains, for star-shaped obstacles and under a unique continuation assumption. Finally, Braz e Silva, Capistrano--Filho, Carvalho and Dos Santos Ferreira \cite{BrCaCaDo} proved local null controllability in $H^1(\mathbb{R}^3)$ of the energy-critical NLS with controls supported outside a ball. To our knowledge, exact controllability of the nonlinear Schr\"odinger equation in an exterior domain, or from a part of the boundary surrounding an obstacle, has not been considered before.

\begin{remark}\label{rem_geo}
If $\Theta$ is star-shaped with respect to some $x^0\in\Theta$, then, with $q(x)=x-x^0$ and $\nu$ the outward unit normal to $\Omega_0$, one has $q\cdot\nu>0$ on $\Gamma_1$ and $q\cdot\nu\leq0$ on $\Gamma_2$, so that \eqref{cond_geo} is the classical partition $\Gamma_1=\{q\cdot\nu>0\}$ of Machtyngier \cite{Mach}. Here, however, $\Theta$ is only assumed to be non-trapping. Equivalently, every generalized bicharacteristic leaves any ball $B_R\supset\Theta$; by compactness, there is then a time $L_R$ such that every unit-speed generalized bicharacteristic meets the region $\omega=\{|x|>R\}$ within time $L_R$. Thus the pair $(\Omega,\omega)$ satisfies the geometric control condition (GCC) of Bardos, Lebeau and Rauch \cite{BLR}; see Figure \ref{fig_geo}(b)--(c). The geometric time $L_R$ is not a restriction on the control time: as is usual for the Schr\"odinger equation, our results hold for every $T>0$. It is through this property, encoded in the resolvent bound \eqref{l9} and in the local smoothing effect \eqref{l12}, that the geometry enters our proofs.
\end{remark}

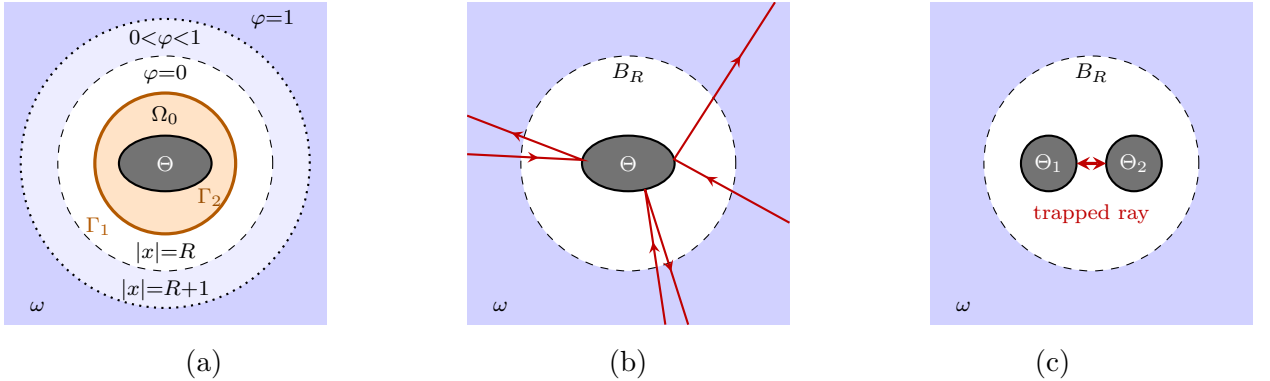
\begin{figure}[!ht]
\centering
\tikzset{
  omegafill/.style={fill=blue!18},
  transfill/.style={fill=blue!7},
  domfill/.style={fill=orange!22},
  obst/.style={fill=black!55,draw=black,thick},
  ray/.style={thick,red!75!black,postaction={decorate},decoration={markings,mark=at position 0.3 with {\arrow{stealth}},mark=at position 0.8 with {\arrow{stealth}}}},
}
\begin{tikzpicture}[scale=0.49,font=\scriptsize]
\clip (-4.35,-4.35) rectangle (4.35,4.35);
\fill[omegafill] (-4.35,-4.35) rectangle (4.35,4.35);
\fill[transfill] (0,0) circle (3.9);
\fill[white] (0,0) circle (2.9);
\fill[domfill] (0,0) circle (1.9);
\draw[obst] (0,0) ellipse (1.25 and 0.75);
\draw[very thick,orange!70!black] (0,0) circle (1.9);
\draw[dashed] (0,0) circle (2.9);
\draw[dotted,thick] (0,0) circle (3.9);
\node[white] at (0,0) {\(\Theta\)};
\node at (0,1.3) {\(\Omega_0\)};
\node[orange!60!black,anchor=north west,inner sep=1pt] at (0.78,-0.6) {\(\Gamma_2\)};
\node[orange!70!black,anchor=north east,inner sep=1pt] at (-135:1.95) {\(\Gamma_1\)};
\node[fill=white,inner sep=1pt] at (0,2.42) {\(\varphi{=}0\)};
\node at (0,3.36) {\(0{<}\varphi{<}1\)};
\node at (2.9,3.9) {\(\varphi{=}1\)};
\node at (-3.45,-3.85) {\(\omega\)};
\node at (0,-2.4) {\(|x|{=}R\)};
\node at (0,-3.4) {\(|x|{=}R{+}1\)};
\end{tikzpicture}\hfill
\begin{tikzpicture}[scale=0.49,font=\scriptsize]
\clip (-4.35,-4.35) rectangle (4.35,4.35);
\fill[omegafill] (-4.35,-4.35) rectangle (4.35,4.35);
\fill[white] (0,0) circle (2.9);
\draw[dashed] (0,0) circle (2.9);
\draw[obst] (0,0) ellipse (1.25 and 0.75);
\node[white] at (0,0) {\(\Theta\)};
\draw[ray] (-4.350,0.250) -- (-1.240,0.095) -- (-4.350,1.291);
\draw[ray] (1.000,-4.350) -- (0.452,-0.699) -- (1.617,-4.350);
\draw[ray] (4.350,-1.600) -- (1.236,0.113) -- (3.959,4.350);
\node at (-3.45,-3.85) {\(\omega\)};
\node[fill=white,inner sep=1pt] at (0,2.45) {\(B_R\)};
\end{tikzpicture}\hfill
\begin{tikzpicture}[scale=0.49,font=\scriptsize]
\clip (-4.35,-4.35) rectangle (4.35,4.35);
\fill[omegafill] (-4.35,-4.35) rectangle (4.35,4.35);
\fill[white] (0,0) circle (2.9);
\draw[dashed] (0,0) circle (2.9);
\draw[obst] (-1.15,0) circle (0.75);
\draw[obst] (1.15,0) circle (0.75);
\draw[very thick,red!75!black,<->,>=stealth] (-0.4,0) -- (0.4,0);
\node[red!75!black,align=center] at (0,-1.35) {trapped\ ray};
\node[white] at (-1.15,0) {\(\Theta_1\)};
\node[white] at (1.15,0) {\(\Theta_2\)};
\node at (-3.45,-3.85) {\(\omega\)};
\node[fill=white,inner sep=1pt] at (0,2.45) {\(B_R\)};
\end{tikzpicture}

\vspace{0.15cm}
\makebox[0.32\textwidth]{(a)}\hfill\makebox[0.32\textwidth]{(b)}\hfill\makebox[0.32\textwidth]{(c)}
\caption{(a) Geometry of Theorem \ref{principal}: the obstacle \(\Theta\), the domain \(\Omega_0=B_{R_0}\setminus\Theta\), the controlled boundary \(\Gamma_1=\partial B_{R_0}\) and the uncontrolled boundary \(\Gamma_2=\partial\Theta\). The internal control \(\varphi h\) of the extended problem \eqref{s1a} acts in \(\omega=\{|x|>R\}\), with \(R=R_0+1\), and vanishes on \(B_R\supset\overline{\Omega_0}\); the cut-off \(\varphi\) is nonconstant only in the annulus \(R<|x|<R+1\). (b) Since \(\Theta\) is non-trapping, every ray of geometric optics, reflected on \(\partial\Theta\), leaves \(B_R\) and enters \(\omega\) within a uniform time: the pair \((\Omega,\omega)\) satisfies the geometric control condition, for every \(R\) with \(\Theta\subset B_R\). (c) For two strictly convex obstacles, the ray bouncing between them never leaves \(B_R\), so the geometric control condition fails for \(\omega=\{|x|>R\}\); this configuration is excluded here (see Section \ref{Sec5}, item (E)).}
\label{fig_geo}
\end{figure}

\subsection{Main results}
Our first result answers Question $\mathcal{B}'$ for the Dirichlet problem.

\begin{theorem}[Exact boundary controllability]
\label{principal}
Let $\Theta$, $\Omega$, $\Omega_0$, $\Gamma_1$, $\Gamma_2$ be as in \eqref{cond_geo}, with $\Theta$ non-trapping and $\Omega$ connected, and let $T>0$, $n\in\{2,3\}$ and $s\in\{1,2\}$. Set $X^s(\Omega_0):=\{u\in H^s(\Omega_0):\ u=0 \text{ on }\Gamma_2\}$. Then there exists $\delta>0$ such that for any $u_0,u_1\in X^s(\Omega_0)$ satisfying $\|u_0\|_{H^s(\Omega_0)}\leq\delta$ and $\|u_1\|_{H^s(\Omega_0)}\leq\delta$, one can find a boundary control $h\in C([0,T];H^{s-1/2}(\Gamma_1))$ such that \eqref{s1}--\eqref{s2} admits a solution $u\in C([0,T];H^s(\Omega_0))$ satisfying
\begin{equation}\label{1.1}
u(x,0)=u_0(x), \qquad u(x,T)=u_1(x) \quad\text{in }\Omega_0.
\end{equation}
Here \eqref{s1} holds in $\mathcal{D}'(\Omega_0\times(0,T))$ and \eqref{s2} holds in the sense of traces.
\end{theorem}

\begin{remark}\label{rem_principal}
The compatibility condition $u_0=u_1=0$ on $\Gamma_2$ is necessary, since no control acts on $\Gamma_2$. Similarly, since $u\in C([0,T];H^s(\Omega_0))$ and $h=u|_{\Gamma_1}$, the control necessarily satisfies $h(0)=u_0|_{\Gamma_1}$ and $h(T)=u_1|_{\Gamma_1}$; the control constructed in the proof (the trace on $\Gamma_1$ of the solution of the extended problem) satisfies these identities. The time $T>0$ is arbitrary, and no star-shapedness of $\Theta$ is required.
\end{remark}

Theorem \ref{principal} is obtained, following \cite{RoZhaJDE}, by extending the data to the exterior domain $\Omega$ and solving an internal control problem there, with a control supported far away from $\Omega_0$. Let $R>0$ with $\Theta\subset B_R(0)$, let $\varphi\in C^\infty(\mathbb{R}^n;[0,1])$ be a cut-off with $\varphi=0$ on $B_R(0)$, $\varphi>0$ for $|x|>R$ and $\varphi=1$ for $|x|\geq R+1$ (see \eqref{phi}), and set $\omega:=\{x\in\Omega:|x|>R\}$. We consider
\begin{equation}\left\{\begin{array}[c]{lll}
iv_t+\Delta v+\lambda|v|^2v=\varphi(x)h(x,t) &  & \text{in }\Omega\times(0,T),\\
v(x,t)=0 && \text{on }\partial\Omega\times(0,T),\\
v(x,0)=v_0(x)   && \text{in }\Omega.
\end{array}
\right.
\label{s1a}
\end{equation}
Here and below, $H^s_D(\Omega)$ denotes the Sobolev scale associated with the Dirichlet Laplacian (see Subsection \ref{sec:linear} for the precise definition); in particular, $H^1_D(\Omega)=H^1_0(\Omega)$ and $H^2_D(\Omega)=H^2(\Omega)\cap H^1_0(\Omega)$. If $u_0,u_1\in X^s(\Omega_0)$ are extended to $v_0,v_T\in H^s_D(\Omega)$ (Lemma \ref{extension}) and $h$ steers the solution of \eqref{s1a} from $v_0$ to $v_T$, then, choosing $R=R_0+1$, the restriction of $v$ to $\Omega_0$ solves \eqref{s1}--\eqref{s2} with the control $v|_{\Gamma_1}$. The core of the paper is therefore the internal controllability of \eqref{s1a}, and first of its linear part
\begin{equation}
\left\{\begin{array}[c]{lll}
iz_t+\Delta z=\varphi(x)h(x,t)&  & \text{in }\Omega\times(0,T),\\
z(x,t)=0& &\text{on }\partial\Omega\times(0,T),\\
z(x,0)=z_0(x)&  &\text{in }\Omega.
\end{array}
\right.  \label{int_lin}
\end{equation}
The following results hold.

\begin{theorem}[Linear controllability]
\label{lin_th}
Assume that $\Theta$ is non-trapping and that $\Omega$ is connected. Let $T>0$ and $s\in[-2,2]$. For any $z_0,z_T\in H^s_D(\Omega)$, one can find a control $h\in C([0,T];H^s_D(\Omega))$ such that the solution $z\in C([0,T];H^s_D(\Omega))$ of \eqref{int_lin} satisfies $z(T)=z_T$.
\end{theorem}

\begin{theorem}[Internal controllability around zero]
\label{nlin_th} Assume that $\Theta$ is non-trapping and that $\Omega$ is connected. Let $T>0$, $n\in\{2,3\}$ and $s\in\{1,2\}$. Then there exists $\delta>0$ such that for any $v_0,v_T\in H^s_D(\Omega)$ satisfying $\|v_0\|_{H^s_D(\Omega)}\leq\delta$ and $\|v_T\|_{H^s_D(\Omega)}\leq\delta$, one can find a control $h\in C([0,T];H^s_D(\Omega))$ such that \eqref{s1a} admits a solution $v\in C([0,T];H^s_D(\Omega))$ with $v(0)=v_0$ and $v(T)=v_T$.
\end{theorem}

We also consider controllability around a nontrivial trajectory. Let $T>0$ and let $w$ satisfy

\noindent
\begin{equation}\tag{W}
\boxed{
  \begin{aligned}
    &w\in C([0,T];H^2_D(\Omega)\cap W^{2,\infty}(\Omega)) \text{ solves} \\
    &\begin{cases} 
      iw_t+\Delta w+\lambda|w|^2w=0 & \text{in } \Omega\times(0,T),\\ 
      w=0 & \text{on }\partial\Omega\times(0,T).
    \end{cases}
  \end{aligned}
}
\end{equation}
Setting $V_1:=2\lambda|w|^2$, $V_2:=\lambda w^2$ and writing $v=w+z$, the linear part of the equation for $z$ is
\begin{equation}\label{linearized_int}
\begin{cases}
iz_t+\Delta z+V_1(x,t)z+V_2(x,t)\overline{z}=\varphi(x)h(x,t) &\text{in }\Omega\times(0,T),\\
z(x,t)=0 &\text{on }\partial\Omega\times(0,T),\\
z(x,0)=z_0(x) &\text{in }\Omega,
\end{cases}
\end{equation}
whose adjoint system, for the real duality $\operatorname{Re}\int_\Omega z\overline{v}\,dx$, is
\begin{equation}\label{adj_w}
\begin{cases}
iv_t+\Delta v+V_1v-V_2\overline v=0,& \ \text{ in }\Omega\times(0,T),\\  v=0,&\ \text{ on }\partial\Omega\times(0,T).
\end{cases}
\end{equation}
As in \cite{Laurent2,CaCoCa}, we assume the following unique continuation property:

\begin{equation}\tag{UCP$_w$}
\boxed{\text{if } v\in C([0,T];H^1_0(\Omega)) \text{ solves } \eqref{adj_w} \text{ and } v=0 \text{ in } \omega\times(0,T), \text{ then } v\equiv0.}
\end{equation}
For $w\equiv0$, \textup{(UCP$_w$)} holds under the sole non-trapping assumption (Lemma \ref{UCP}). With this hypothesis at hand, our results around nontrivial trajectories read as follows.
\begin{theorem}[Linearized controllability around $w$]\label{lin_ar}
Assume that $\Theta$ is non-trapping, that $\Omega$ is connected, and that \textup{(W)} and \textup{(UCP$_w$)} hold. Let $s\in\{1,2\}$. For any $z_0,z_T\in H^s_D(\Omega)$ one can find a control $h\in C([0,T];H^s_D(\Omega))$ such that the solution $z\in C([0,T];H^s_D(\Omega))$ of \eqref{linearized_int} satisfies $z(T)=z_T$.
\end{theorem}

\begin{theorem}[Internal controllability around $w$]\label{nlin_ar}
Assume that $\Theta$ is non-trapping, that $\Omega$ is connected, and that \textup{(W)} and \textup{(UCP$_w$)} hold. Let $n\in\{2,3\}$ and $s\in\{1,2\}$. Then there exists $\delta>0$ such that for any $v_0,v_T\in H^s_D(\Omega)$ satisfying
$$
\|v_0-w(0)\|_{H^s_D(\Omega)}\leq\delta\quad\text{and}\quad\|v_T-w(T)\|_{H^s_D(\Omega)}\leq\delta,
$$
one can find a control $h\in C([0,T];H^s_D(\Omega))$ such that \eqref{s1a} admits a solution $v\in C([0,T];H^s_D(\Omega))$ with
\begin{equation}\label{1.5}
v(x,0)=v_0(x), \qquad v(x,T)=v_T(x)\quad\text{in }\Omega.
\end{equation}
\end{theorem}

Finally, the method applies to the quintic nonlinearity near the zero solution.

\begin{corollary}[Quintic nonlinearity]\label{quintic}
Assume that $\Theta$ is non-trapping and that $\Omega$ is connected, and let $T>0$, $n\in\{2,3\}$ and $s\in\{1,2\}$. Then the conclusions of Theorems \ref{principal} and \ref{nlin_th} remain valid when the cubic terms $\lambda|u|^2u$ in \eqref{s1} and $\lambda|v|^2v$ in \eqref{s1a} are replaced by $\lambda|u|^4u$ and $\lambda|v|^4v$, respectively.
\end{corollary}

For $n=3$ and $s=1$, the quintic nonlinearity is energy-critical; compare with \cite{BrCaCaDo}, where local null controllability is proved on $\mathbb{R}^3$.

\subsection{Novelty and main ingredients}
To the best of our knowledge, Theorems \ref{principal}--\ref{nlin_ar} and Corollary \ref{quintic} are the first exact controllability results for the nonlinear Schr\"odinger equation in an exterior domain, and Theorem \ref{principal} is the first exact controllability result for a nonlinear Schr\"odinger equation with controls acting only on the outer part of the boundary of a domain surrounding an obstacle. Let us point out the main features of the approach.

\begin{itemize}
\item[(i)] \emph{Only the non-trapping assumption is needed.} The usual way of proving observability for controls supported outside a ball is the multiplier $q(x)=x$, as in \cite{RoZhaJDE,BrCaCaDo} on $\mathbb{R}^n$. In an exterior domain this produces the boundary term $$\frac12\int_{\partial\Omega\times(0,T)}(q\cdot\nu)|\partial_\nu v|^2,$$ whose favorable sign requires $\Theta$ to be star-shaped. (For data in $H^1_0(\Omega)$, this normal trace is well defined in $L^2(\partial\Omega\times(0,T))$, by the usual hidden regularity estimate obtained with a compactly supported multiplier equal to $\nu$ on $\partial\Omega$; the obstruction is its sign, not its existence.) Our argument instead relies on the local smoothing effect of Burq, G\'erard and Tzvetkov \cite{BGT1}, which holds for every non-trapping obstacle: it turns the contribution of the region near the obstacle into a lower order term, which is then removed by a compactness--uniqueness argument (see Subsection \ref{sec:linear} for details). This strategy was used at the $L^2$ level by Bortot and Cavalcanti \cite{BoCa} to prove the exponential stabilization of the damped equation; for linear damping, the decay proved there is equivalent to an $L^2$ observability inequality from $\omega$ for some $T>0$ \cite{Liu}. The new point here is that the argument is carried out at every level $H^\sigma_D(\Omega)$, $\sigma\in[-2,2]$, including the negative ones needed for the HUM method in $H^1_0(\Omega)$ and in $H^2(\Omega)\cap H^1_0(\Omega)$; at each of these levels, the proof of the observability inequality uses no integration by parts up to $\partial\Theta$ and no normal boundary trace. (The boundary control result, Theorem \ref{principal}, uses of course the Dirichlet trace of the solution on the outer sphere $\Gamma_1$.)
\item[(ii)] \emph{Unique continuation by a spectral argument.} For the free equation, unique continuation from $\omega$ is reduced to the absence of eigenfunctions of $\Delta_D$ vanishing outside a ball (Lemma \ref{UCP}). This replaces the Paley--Wiener--Schwartz argument used on $\mathbb{R}^3$ in the recent work \cite[Lemma 3.1]{BrCaCaDo}, which relies on the explicit kernel of $e^{it\Delta}$ and has no analogue on $\Omega$. For $L^2$-valued solutions, unique continuation can also be obtained from Holmgren's theorem, as in \cite{BoCa}; the spectral argument has the advantage of applying directly at every level of the scale $H^\sigma_D(\Omega)$, $\sigma\in[-2,2]$.
\item[(iii)] \emph{Arbitrary control time and the whole scale $H^s_D(\Omega)$, $s\in[-2,2]$.} The linear results hold for every $T>0$ and every $s\in[-2,2]$, with controls in $C([0,T];H^s_D(\Omega))$, and no commutator between $\varphi$ and fractional powers of $\Delta_D$ is needed (see Subsection \ref{sec:linear}). The nonlinear results are obtained by a perturbation argument in which the smallness comes only from the data, using, in the energy space, the Strichartz estimates with loss of \cite{BGT1} in dimension two and the scale-invariant $L^4_tL^\infty_x$ estimate of Blair, Smith and Sogge \cite{BSS} in dimension three, and the algebra property of $H^2$ in dimensions two and three.
\item[(iv)] \emph{Nontrivial trajectories.} The same compactness--uniqueness scheme applies to the linearized equation around a solution $w$. Local smoothing also allows us to show that solutions invisible from $\omega$ are automatically in $C([0,T];H^1_0(\Omega))$ (see Lemma \ref{regularity_w}), so that unique continuation needs to be assumed only at the energy level, as in \cite{Laurent2,CaCoCa}.
\end{itemize}

\subsection{Organization of the paper}
Section \ref{Sec2} collects the linear estimates in exterior domains due to Burq, G\'erard and Tzvetkov, and an $L^4_tL^\infty_x$ estimate of Blair, Smith and Sogge. In Section \ref{Sec3} we prove the observability, and consequently, the controllability of the linear equation (Theorem \ref{lin_th}), the internal controllability of the NLS around zero (Theorem \ref{nlin_th}), the boundary controllability (Theorem \ref{principal}) and the quintic case (Corollary \ref{quintic}). Section \ref{Sec4} deals with the control around nontrivial trajectories (Theorems \ref{lin_ar} and \ref{nlin_ar}). Finally, Section \ref{Sec5} discusses the scope of the results and some open problems.

\section{Linear estimates in exterior domains\label{Sec2}}
In this section we collect the linear estimates for the Schr\"odinger group in the exterior of a non-trapping obstacle on which our analysis relies. Except for Proposition \ref{BSS_est}, due to Blair, Smith and Sogge \cite{BSS}, all of them are due to Burq, G\'erard and Tzvetkov \cite{BGT1}; we state them in the form used below and refer to \cite{BSS,BGT1} for the proofs.

Let $\Theta\subset\mathbb{R}^n$, $n\geq2$, be a nonempty compact set with smooth boundary, $\Omega=\mathbb{R}^n\setminus\Theta$, and let $\Delta_D$ be the Dirichlet Laplacian on $\Omega$, with domain $H^2(\Omega)\cap H^1_0(\Omega)$. All cut-off functions used in this paper are real-valued. For $s\in\mathbb{R}$ we denote by $H^s_D(\Omega)$ the scale of spaces associated with $(1-\Delta_D)^{s/2}$ (see Subsection \ref{sec:linear} for the precise definition)\footnote{Recall that $H^s_D(\Omega)=H^s(\Omega)$ for $s\in[0,\frac12)$, that $H^s_D(\Omega)=H^s(\Omega)\cap H^1_0(\Omega)$ for $s\in[1,2]$ (in particular $H^1_D(\Omega)=H^1_0(\Omega)$ and $H^2_D(\Omega)=H^2(\Omega)\cap H^1_0(\Omega)$), and that $H^{-s}_D(\Omega)$ is the dual of $H^s_D(\Omega)$.}.

\begin{proposition}[{Cut-off resolvent, \cite[Propositions 2.2, 2.4 and 2.5]{BGT1}}]\label{p4}
Assume that $\Theta$ is non-trapping. Then for every $\chi\in C^\infty_0(\mathbb{R}^n)$ and every $s\in[-1,1]$ there exists $C>0$ such that, for every $\kappa\in\mathbb{R}$ and $0<\varepsilon\ll1$,
\begin{equation}\label{l9}
\|\chi(-\Delta_D-(\kappa\pm i\varepsilon)^2)^{-1}\chi\|_{H^s_D(\Omega)\to H^{s+1}_D(\Omega)}\leq C .
\end{equation}
\end{proposition}

We state \eqref{l9} only for $s\in[-1,1]$, which is the range used below. For $s\geq\frac32$ the spaces $H^{s+1}_D(\Omega)$ encode further compatibility conditions on $\partial\Omega$ (for instance $\Delta u=0$ on $\partial\Omega$ for $u\in H^4_D(\Omega)$), which are not preserved by multiplication by a general cut-off $\chi$; the range $[-2,-1]$ will be obtained in Lemma \ref{smoothing_neg} by duality from the case $s=1$. The high frequency part of \eqref{l9} is a consequence of the non-trapping assumption, while the low frequency part holds in every dimension $n\geq2$ for the Dirichlet Laplacian because $\Theta\neq\emptyset$ \cite[Proposition 2.4]{BGT1}. In \cite[Proposition 2.5]{BGT1} the cut-off is assumed to be nonnegative; the general case follows by writing $\chi=\chi\chi_1$, with $0\leq\chi_1\in C^\infty_0(\mathbb{R}^n)$ equal to $1$ on $\operatorname{supp}\chi$, since multiplication by $\chi$ is bounded on $H^s_D(\Omega)$ and on $H^{s+1}_D(\Omega)$ for $s\in[-1,1]$.

\begin{proposition}[{Local smoothing, \cite[Proposition 2.7]{BGT1}}]\label{p5}
Assume that $\Theta$ is non-trapping. For every $\chi\in C^\infty_0(\mathbb{R}^n)$ there exists $C>0$ such that, for every $T>0$,
\begin{equation}\label{l11}
\|\chi u\|_{L^2(0,T;H^{s+1}_D(\Omega))}\leq C\|\chi f\|_{L^2(0,T;H^s_D(\Omega))},\quad s\in[-1,1],\end{equation}
with 
$$u(t)=\int_0^te^{i(t-r)\Delta_D}\chi f(r)\,dr,$$
and
\begin{equation}\label{l12}
\|\chi e^{it\Delta_D}v_0\|_{L^2(0,T;H^{s+1/2}_D(\Omega))}\leq C\|v_0\|_{H^s_D(\Omega)},\qquad s\in[0,1].
\end{equation}
\end{proposition}

\begin{proposition}[{Strichartz estimates, \cite[Propositions 2.10 and 2.11]{BGT1}}]\label{st1}
Assume that $\Theta$ is non-trapping. Let $T>0$, $s\in[0,1]$ and let $(p,q)$, $p>2$, satisfy $\frac2p+\frac nq=\frac n2$. Then, for $\chi\in C^\infty_0(\mathbb{R}^n)$ with $\chi=1$ near $\Theta$,
\begin{equation}\label{se1}
\|(1-\chi)e^{it\Delta_D}u_0\|_{L^p(0,T;W^{s,q}(\Omega))}\leq C\|u_0\|_{H^s_D(\Omega)},
\end{equation}
and, globally in $\Omega$, with a loss of $1/p$ derivative,
\begin{equation}\label{main_1}
\|e^{it\Delta_D}u_0\|_{L^p(0,T;W^{s,q}(\Omega))}\leq C\|u_0\|_{H^{s+1/p}_D(\Omega)} .
\end{equation}
\end{proposition}

In dimension three we shall also use the following scale-invariant estimate.

\begin{proposition}[{\cite[Section 5]{BSS}}]\label{BSS_est}
Let $n=3$ and assume that $\Theta$ is non-trapping. For every $T>0$ there exists $C>0$ such that
\begin{equation}\label{L4Linf3}
\|e^{it\Delta_D}u_0\|_{L^4(0,T;L^\infty(\Omega))}\leq C\|u_0\|_{H^1_0(\Omega)},\qquad\forall\,u_0\in H^1_0(\Omega).
\end{equation}
\end{proposition}

\begin{remark}\label{rem_sec2}
The loss in \eqref{main_1} comes from the region near the obstacle, where one only uses \eqref{l12} and Sobolev embeddings. Scale-invariant Strichartz estimates, in a restricted range of exponents, hold for general non-trapping obstacles \cite[Theorem 1.1]{BSS}; estimate \eqref{L4Linf3} is deduced from them in \cite[Section 5]{BSS} by a Littlewood--Paley argument, and is used there to give a simple proof of the well-posedness of the energy-critical equation in $H^1_0(\Omega)$, first obtained by Ivanovici and Planchon \cite{IP}. The case of Neumann boundary conditions can be handled similarly \cite[Remark 1.2]{BGT1}, with constants depending on $T$ in the local smoothing estimates \cite[Remarks 2.3 and 2.9]{BGT1}.
\end{remark}

\section{Internal controllability around the zero solution\label{Sec3}}
\subsection{Linear result}\label{sec:linear} We first fix the assumptions and the notation used in this subsection. Throughout, we assume:
\begin{equation}\tag{H}
\boxed{
  \parbox{0.82\textwidth}{
    $n\geq 2$, $\Theta\subset\mathbb{R}^n$ is a nonempty compact set with smooth boundary, $\Theta$ is \emph{non-trapping}, and $\Omega:=\mathbb{R}^n\setminus\Theta$ is connected.
  }
}
\end{equation}

Fix $R>0$ with $\Theta\subset B_R(0)$. Let $\varphi\in C^\infty(\mathbb{R}^n;[0,1])$ be such that
\begin{equation}\label{phi}
\varphi(x)=0 \ \text{ if } |x|\leq R,\qquad \varphi(x)>0\ \text{ if } |x|>R,\qquad \varphi(x)=1\ \text{ if } |x|\geq R+1,
\end{equation}
and set $\chi:=1-\varphi\in C^\infty_0(\mathbb{R}^n)$, so that $\operatorname{supp}\chi\subset \overline{B_{R+1}(0)}$. The control region is
\begin{equation}\label{omega}
\omega:=\{x\in\Omega:\ |x|>R\}=\{x\in\Omega:\ \varphi(x)>0\}.
\end{equation}
We also fix $\psi\in C_0^\infty(\mathbb{R}^n;[0,1])$ with $\psi=1$ on $B_{R+2}(0)$ and $\operatorname{supp}\psi\subset B_{R+3}(0)$. In particular
\begin{equation}\label{psi}
0\leq\chi\leq\psi \quad\text{and}\quad \psi=1 \text{ on } \operatorname{supp}\nabla\varphi .
\end{equation}
We write $\Omega':=\Omega\cap B_{R+3}(0)$, which is a bounded open set with smooth boundary $\partial\Theta\cup\partial B_{R+3}(0)$.

Let $\Delta_D$ be the Dirichlet Laplacian on $\Omega$. For $\sigma\in\mathbb{R}$, $H^\sigma_D(\Omega)$ denotes the completion of $\mathcal{D}((1-\Delta_D)^{\infty})$ for the norm $\|u\|_{H^\sigma_D}:=\|(1-\Delta_D)^{\sigma/2}u\|_{L^2(\Omega)}$; for $\sigma\geq 0$ this is the domain of $(1-\Delta_D)^{\sigma/2}$, $H^1_D(\Omega)=H^1_0(\Omega)$, and $H^{-\sigma}_D(\Omega)$ is identified with the (anti)dual of $H^\sigma_D(\Omega)$ through the sesquilinear pairing $$\langle f,g\rangle=\int_\Omega f\overline{g}\,dx.$$ We set
$$
\mathcal{R}:=(1-\Delta_D)^{-1},
$$
which is an isometric isomorphism from $H^{\sigma}_D(\Omega)$ onto $H^{\sigma+2}_D(\Omega)$ for every $\sigma\in\mathbb{R}$. The group $e^{it\Delta_D}$ is unitary on every $H^\sigma_D(\Omega)$ and commutes with $\mathcal{R}$; in particular
\begin{equation}\label{H1}
\|e^{it\Delta_D}v_0\|_{H^\sigma_D(\Omega)}=\|v_0\|_{H^\sigma_D(\Omega)},\qquad t\in\mathbb{R},\ \sigma\in\mathbb{R}.
\end{equation} Since $\varphi$, $\chi$ and $\psi$ are real-valued, smooth, bounded with bounded derivatives, and constant near $\partial\Omega$, multiplication by any of them is bounded on $H^\sigma_D(\Omega)$ for $\sigma\in[0,2]$ (interpolation between $L^2$, $H^1_0$ and $H^2\cap H^1_0$) and, by duality, for $\sigma\in[-2,0]$.

For $v_0\in H^\sigma_D(\Omega)$ we always denote $v(t):=e^{it\Delta_D}v_0$, that is, $v$ solves
\begin{equation}\label{adj}
\begin{cases}
iv_t+\Delta v=0, & x\in\Omega,\ t\in\mathbb{R},\\
v(x,t)=0, & x\in\partial\Omega,\ t\in\mathbb{R},\\
v(x,0)=v_0(x), & x\in\Omega .
\end{cases}
\end{equation}

In this subsection, the geometry of $\Theta$ enters only through the following local smoothing effect, due to Burq, G\'erard and Tzvetkov.

\begin{proposition}[{\cite[Proposition 2.7]{BGT1}}]\label{smoothing}
Assume \textup{(H)}. For every $\eta\in C^\infty_0(\mathbb{R}^n)$ and every $s\in[0,1]$ there exists $C>0$ such that, for every $T>0$,
\begin{equation}\label{smooth_est}
\int_0^T\|\eta\, e^{it\Delta_D}v_0\|^2_{H^{s+1/2}_D(\Omega)}\,dt\leq C\|v_0\|^2_{H^s_D(\Omega)},\qquad \forall\, v_0\in H^s_D(\Omega).
\end{equation}
\end{proposition}

We shall work with the set of indices
$$
\mathcal{I}:=[-2,2),
$$
and, for $\sigma\in\mathcal I$, we put $k(\sigma):=0$ if $\sigma\geq 0$, $k(\sigma):=1$ if $\sigma<0$, so that $\tau:=\sigma+2k(\sigma)\in[0,2)$. We also put $\rho(\sigma):=\sigma+1$ if $\sigma\in(-1,0)$ and $\rho(\sigma):=0$ otherwise; thus $\rho(\sigma)=\tau-1\in(0,1)$ if $\sigma\in(-1,0)$. For $v_0\in H^\sigma_D(\Omega)$ we define the lower order quantity
\begin{equation}\label{Qsigma}
Q_\sigma(v_0):=\int_0^T\big\|\psi\,\mathcal{R}^{k(\sigma)}v(t)\big\|^2_{H^{\rho(\sigma)}_D(\Omega)}\,dt .
\end{equation}
Note that $$\mathcal{R}^{k(\sigma)}v(t)=e^{it\Delta_D}\mathcal{R}^{k(\sigma)}v_0$$ and $$\|\mathcal R^{k(\sigma)}v_0\|_{H^\tau_D}=\|v_0\|_{H^\sigma_D}.$$ Except for $\sigma\in(-1,0)$, $Q_\sigma$ is an $L^2$ quantity; the norm of $H^{\rho(\sigma)}_D(\Omega)$, which is still weaker than that of $H^\tau_D(\Omega)$, is only needed for $\sigma\in(-1,0)$ (see the proof of Lemma \ref{relaxed}). The endpoint $\sigma=2$ will be deduced from the case $\sigma=0$ in Proposition \ref{obs}.

We divide the proof of Theorem \ref{lin_th} into five steps.
\vspace{0.1cm}

\noindent\textbf{Step 1.} \textit{Compactness of $Q_\sigma$.}

\begin{lemma}\label{compact}
Assume \textup{(H)}, let $T>0$ and $\sigma\in\mathcal I$. If $v_{j,0}\rightharpoonup v_0$ weakly in $H^\sigma_D(\Omega)$, then $$Q_\sigma(v_{j,0}-v_0)\to0.$$
\end{lemma}

\begin{proof}
Set $k=k(\sigma)$, $\tau=\sigma+2k\in[0,2)$, $\rho=\rho(\sigma)$, $y_{j,0}:=\mathcal R^k(v_{j,0}-v_0)$ and $y_j(t):=e^{it\Delta_D}y_{j,0}$. Then $y_{j,0}\rightharpoonup0$ weakly in $H^\tau_D(\Omega)$, hence weakly in $L^2(\Omega)$, and $\|y_{j,0}\|_{H^\tau_D}\leq M$ for some $M$; in particular $\|y_{j,0}\|_{L^2}\leq M$.

\medskip
\noindent\textit{(a) Bounds.} By Proposition \ref{smoothing} with $s=0$ and $\eta=\psi$,
$$
\int_0^T\|\psi y_j(t)\|^2_{H^{1/2}(\Omega')}\,dt\leq C\int_0^T\|\psi y_j(t)\|^2_{H^{1/2}_D(\Omega)}\,dt\leq CM^2 .
$$
Moreover $\partial_t(\psi y_j)=i\psi\Delta_D y_j$, and for $\zeta\in C^\infty_0(\Omega')$ we have $\psi\zeta\in H^2(\Omega)\cap H^1_0(\Omega)$, so that
$$
|\langle \psi\Delta_Dy_j(t),\zeta\rangle|=|\langle y_j(t),\Delta(\psi\zeta)\rangle|\leq C\|y_j(t)\|_{L^2}\|\zeta\|_{H^2(\Omega')}\leq CM\|\zeta\|_{H^2(\Omega')} .
$$
Hence $(\partial_t(\psi y_j))_j$ is bounded in $L^\infty(0,T;H^{-2}(\Omega'))$, where $H^{-2}(\Omega'):=(H^2_0(\Omega'))'$.

\medskip
\noindent\textit{(b) Compactness.} Since $\Omega'$ is bounded and smooth, $H^{1/2}(\Omega')\hookrightarrow L^2(\Omega')$ is compact and $L^2(\Omega')\hookrightarrow H^{-2}(\Omega')$ is continuous. By the Aubin--Lions--Simon lemma \cite{Simon}, $(\psi y_j)_j$ is relatively compact in $L^2(0,T;L^2(\Omega'))$. As $\operatorname{supp}\psi\subset B_{R+3}(0)$, this is the same as relative compactness in $L^2(0,T;L^2(\Omega))$.

\medskip
\noindent\textit{(c) Identification of the limit.} For $g\in L^2(0,T;L^2(\Omega))$,
$$
\int_0^T\langle\psi y_j(t),g(t)\rangle\,dt=\Big\langle y_{j,0},\int_0^Te^{-it\Delta_D}(\psi g(t))\,dt\Big\rangle\longrightarrow0,
$$
since $$\int_0^Te^{-it\Delta_D}(\psi g)\,dt\in L^2(\Omega).$$ Thus $\psi y_j\rightharpoonup0$ weakly in $L^2(0,T;L^2(\Omega))$, and by (b) every subsequence has a further subsequence converging strongly to $0$. Therefore $\psi y_j\to0$ in $L^2(0,T;L^2(\Omega))$, which is the claim if $\rho=0$.

\medskip
\noindent\textit{(d) The case $\rho>0$.} Then $\sigma\in(-1,0)$, $\tau\in(1,2)$ and $\rho=\tau-1$. Since multiplication by $\psi$ is bounded on $H^\tau_D(\Omega)$ and $e^{it\Delta_D}$ is unitary on this space, $\|\psi y_j(t)\|_{H^\tau_D}\leq CM$ for all $t$ and $j$. The spectral theorem and H\"older's inequality give $\|f\|_{H^\rho_D}\leq\|f\|^{1-\rho/\tau}_{L^2}\|f\|^{\rho/\tau}_{H^\tau_D}$, and H\"older's inequality in time yields
$$
Q_\sigma(v_{j,0}-v_0)=\int_0^T\|\psi y_j(t)\|^2_{H^\rho_D}\,dt\leq(CM)^{\frac{2\rho}\tau}\,T^{\frac\rho\tau}\Big(\int_0^T\|\psi y_j(t)\|^2_{L^2}\,dt\Big)^{1-\frac\rho\tau}\longrightarrow0 .
$$
\end{proof}

\vspace{0.1cm}

\noindent\textbf{Step 2.} \textit{Observability up to a compact term.}

\begin{lemma}\label{relaxed}
Assume \textup{(H)} and let $T>0$, $\sigma\in\mathcal I$. There exists $C>0$ such that
\begin{equation}\label{relaxed_est}
\|v_0\|^2_{H^\sigma_D(\Omega)}\leq C\left(\int_0^T\|\varphi v(t)\|^2_{H^\sigma_D(\Omega)}\,dt+Q_\sigma(v_0)\right),\qquad\forall\,v_0\in H^\sigma_D(\Omega).
\end{equation}
\end{lemma}

\begin{proof} We split the proof in two cases.

\medskip
\noindent\textit{1. Case $\sigma=\tau\in[0,2)$.}
\vspace{0.1cm}

By unitarity of $e^{it\Delta_D}$ on $H^\tau_D(\Omega)$ and $v=\varphi v+\chi v$,
\begin{equation}\label{rel1}
T\|v_0\|^2_{H^\tau_D}=\int_0^T\|v(t)\|^2_{H^\tau_D}\,dt\leq2\int_0^T\|\varphi v(t)\|^2_{H^\tau_D}\,dt+2\int_0^T\|\chi v(t)\|^2_{H^\tau_D}\,dt .
\end{equation}
If $\tau=0$, then by \eqref{psi} $\|\chi v\|_{L^2}\leq\|\psi v\|_{L^2}$, and \eqref{rel1} is \eqref{relaxed_est}. Let $\tau\in(0,2)$ and set
$$
\theta:=0\ \text{ if }\tau\in(0,\tfrac12),\qquad \theta:=\tau-\tfrac12\ \text{ if }\tau\in[\tfrac12,\tfrac32],\qquad \theta:=2\tau-2\ \text{ if }\tau\in(\tfrac32,2),
$$
so that $0\leq\theta<\tau$. We claim that
\begin{equation}\label{rel_chi}
\int_0^T\|\chi v(t)\|^2_{H^\tau_D}\,dt\leq C_1\|v_0\|^2_{H^\theta_D}.
\end{equation}
If $\tau\in[\frac12,\frac32]$ we apply Proposition \ref{smoothing} with $s=\tau-\frac12$ and $\eta=\chi$; if $\tau<\frac12$ we use $\|\chi v\|_{H^\tau_D}\leq\|\chi v\|_{H^{1/2}_D}$ and Proposition \ref{smoothing} with $s=0$. If $\tau\in(\frac32,2)$, we interpolate between the estimates
$$
\|\chi e^{it\Delta_D}v_0\|_{L^2(0,T;H^{3/2}_D)}\leq C\|v_0\|_{H^1_D},\qquad\|\chi e^{it\Delta_D}v_0\|_{L^2(0,T;H^2_D)}\leq C\|v_0\|_{H^2_D},
$$
the first one being Proposition \ref{smoothing} with $s=1$, and the second one a consequence of the unitarity of $e^{it\Delta_D}$ on $H^2_D(\Omega)$ and of the boundedness of multiplication by $\chi$ on this space. Complex interpolation with parameter $a=2\tau-3\in(0,1)$ gives \eqref{rel_chi}, since $[H^1_D,H^2_D]_a=H^{2\tau-2}_D$ and $[H^{3/2}_D,H^2_D]_a=H^{\tau}_D$.

Since $0\leq\theta<\tau$, the spectral theorem gives, for every $\varepsilon>0$, $$\|v_0\|^2_{H^\theta_D}\leq\varepsilon\|v_0\|^2_{H^\tau_D}+C_\varepsilon\|v_0\|^2_{L^2}.$$ Choosing $2C_1\varepsilon=T/2$ in \eqref{rel1} we obtain
$$
\frac T2\|v_0\|^2_{H^\tau_D}\leq2\int_0^T\|\varphi v(t)\|^2_{H^\tau_D}\,dt+2C_1C_\varepsilon\|v_0\|^2_{L^2}.
$$
Finally, by unitarity in $L^2$, $v=\varphi v+\chi v$ and \eqref{psi},
$$
\|v_0\|^2_{L^2}=\frac1T\int_0^T\|v(t)\|^2_{L^2}\,dt\leq\frac2T\left(\int_0^T\|\varphi v(t)\|^2_{H^\tau_D}\,dt+\int_0^T\|\psi v(t)\|^2_{L^2}\,dt\right),
$$
which yields \eqref{relaxed_est}, since $\rho(\sigma)=0$ in this case.

\medskip
\noindent\textit{2. Case $\sigma=\tau-2\in[-2,0)$.}
\vspace{0.1cm}

 Put $y:=\mathcal Rv=e^{it\Delta_D}y_0$, $y_0:=\mathcal Rv_0\in H^\tau_D(\Omega)$, so that $\|y_0\|_{H^\tau_D}=\|v_0\|_{H^\sigma_D}$. By the first case, applied to $y_0$,
\begin{equation}\label{rel2}
\|v_0\|^2_{H^\sigma_D}\leq C\left(\int_0^T\|\varphi y(t)\|^2_{H^\tau_D}\,dt+\int_0^T\|\psi y(t)\|^2_{L^2}\,dt\right).
\end{equation}
We compare $\varphi y=\varphi\mathcal Rv$ with $\mathcal R(\varphi v)$. Since $\varphi$ vanishes near $\partial\Omega$, for $u\in\mathcal D(\Delta_D)$ we have $\varphi u\in\mathcal D(\Delta_D)$ and $(1-\Delta_D)(\varphi u)-\varphi(1-\Delta_D)u=-[\Delta,\varphi]u$, where
$$
[\Delta,\varphi]u:=(\Delta\varphi)u+2\nabla\varphi\cdot\nabla u=2\operatorname{div}(u\nabla\varphi)-(\Delta\varphi)u .
$$
Applying $\mathcal R$ on both sides with $u=\mathcal R f$ and using density, we get $$\varphi\mathcal R f-\mathcal R(\varphi f)=-\mathcal R[\Delta,\varphi]\mathcal R f,$$ for $f\in H^\sigma_D(\Omega)$. For $u\in L^2(\Omega)$ and $\zeta\in H^1_0(\Omega)$,
$$
|\langle[\Delta,\varphi]u,\zeta\rangle|=\Big|-2\int_\Omega u\nabla\varphi\cdot\nabla\overline\zeta\,dx-\int_\Omega(\Delta\varphi)u\overline\zeta\,dx\Big|\leq C\|u\|_{L^2}\|\zeta\|_{H^1_0},
$$
so that $[\Delta,\varphi]$ is bounded from $L^2(\Omega)$ to $H^{-1}_D(\Omega)$. Being a first order operator with smooth coefficients supported in $\operatorname{supp}\nabla\varphi$, it is also bounded from $H^1_D(\Omega)$ to $L^2(\Omega)$, hence, by interpolation, from $H^\alpha_D(\Omega)$ to $H^{\alpha-1}_D(\Omega)$ for every $\alpha\in[0,1]$. Moreover $$[\Delta,\varphi]u=[\Delta,\varphi](\psi u),$$ because $\psi=1$ in a neighbourhood of $\operatorname{supp}\nabla\varphi$ by \eqref{psi}. If $\tau\in[0,1]$, then $\sigma\leq-1$ and $$\|[\Delta,\varphi]y\|_{H^\sigma_D}\leq\|[\Delta,\varphi](\psi y)\|_{H^{-1}_D}\leq C\|\psi y\|_{L^2};$$ if $\tau\in(1,2)$, we take $\alpha=\tau-1=\rho(\sigma)$ and obtain $\|[\Delta,\varphi]y\|_{H^\sigma_D}\leq C\|\psi y\|_{H^{\tau-1}_D}$. In both cases $\|[\Delta,\varphi]y\|_{H^\sigma_D}\leq C\|\psi y\|_{H^{\rho(\sigma)}_D}$, and since $\mathcal R:H^\sigma_D\to H^\tau_D$ is an isometry,
$$
\|\varphi y(t)\|_{H^\tau_D}\leq\|\varphi v(t)\|_{H^\sigma_D}+C\|\psi y(t)\|_{H^{\rho(\sigma)}_D}.
$$
Plugging this into \eqref{rel2}, and using $\|\psi y\|_{L^2}\leq\|\psi y\|_{H^{\rho(\sigma)}_D}$ and $$Q_\sigma(v_0)=\int_0^T\|\psi y(t)\|^2_{H^{\rho(\sigma)}_D}dt,$$ gives \eqref{relaxed_est}.
\end{proof}

\vspace{0.1cm}

\noindent\textbf{Step 3.} \textit{Unique continuation.}

For $T>0$ and $\sigma\in[-2,2]$, let
$$
\mathcal N^\sigma_T:=\left\{v_0\in H^\sigma_D(\Omega):\ \varphi\, e^{it\Delta_D}v_0=0\ \text{ for all } t\in[0,T]\right\}.
$$
It is a closed complex subspace of $H^\sigma_D(\Omega)$; moreover $\mathcal N^\sigma_{T'}\subset\mathcal N^\sigma_T$ whenever $T\leq T'$, and $\mathcal N^{\sigma'}_T\subset\mathcal N^\sigma_T$ whenever $\sigma\leq\sigma'$.

\begin{lemma}\label{UCP}
Assume \textup{(H)}. Then $\mathcal N^\sigma_T=\{0\}$ for every $T>0$ and every $\sigma\in[-2,2]$.
\end{lemma}

\begin{proof}
Since $\mathcal N^\sigma_T\subset\mathcal N^{-2}_T$ for every $\sigma\in[-2,2]$, it would be enough to consider $\sigma=-2$; we give the argument for an arbitrary $\sigma\in\mathcal I$.

\textit{(i) $\mathcal N^\sigma_T$ is finite dimensional.} For $v_0\in\mathcal N^\sigma_T$, \eqref{relaxed_est} reads $$\|v_0\|^2_{H^\sigma_D}\leq C\,Q_\sigma(v_0).$$ Let $(v_{j,0})$ be a bounded sequence in $\mathcal N^\sigma_T$. Extracting a subsequence, $v_{j,0}\rightharpoonup v_0$ weakly in $H^\sigma_D(\Omega)$, and $v_0\in\mathcal N^\sigma_T$ because this subspace is weakly closed. Applying \eqref{relaxed_est} to $v_{j,0}-v_0\in\mathcal N^\sigma_T$ and Lemma \ref{compact},
$$
\|v_{j,0}-v_0\|^2_{H^\sigma_D}\leq C\,Q_\sigma(v_{j,0}-v_0)\longrightarrow0 .
$$
Thus the closed unit ball of $\mathcal N^\sigma_T$ is compact and $\dim\mathcal N^\sigma_T<\infty$.

\smallskip
\textit{(ii) Invariance under $\Delta_D$.} Fix $T_0>0$. The map $T\mapsto d(T):=\dim\mathcal N^\sigma_T$ is non-increasing and, by (i), takes finitely many values on $[T_0/2,T_0]$. Hence there are $T_0/2\leq a<b\leq T_0$ such that $d$ is constant on $(a,b)$; since the spaces are nested and finite dimensional, $\mathcal N^\sigma_T=\mathcal N^\sigma_{T'}=:\mathcal N$ for all $T,T'\in(a,b)$. Fix $T_1\in(a,b)$ and $0<\varepsilon<b-T_1$. If $v_0\in\mathcal N=\mathcal N^\sigma_{T_1+\varepsilon}$, then $$\varphi e^{it\Delta_D}(e^{i\varepsilon\Delta_D}v_0)=\varphi e^{i(t+\varepsilon)\Delta_D}v_0=0,$$ for $t\in[0,T_1]$, i.e. $e^{i\varepsilon\Delta_D}v_0\in\mathcal N^\sigma_{T_1}=\mathcal N$. Consequently
$$
\frac{e^{i\varepsilon\Delta_D}v_0-v_0}{\varepsilon}\in\mathcal N\quad\text{and}\quad \frac{e^{i\varepsilon\Delta_D}v_0-v_0}{\varepsilon}\xrightarrow[\varepsilon\to0^+]{}i\Delta_Dv_0\ \text{ in } H^{\sigma-2}_D(\Omega).
$$
Being finite dimensional, $\mathcal N$ is closed in $H^{\sigma-2}_D(\Omega)$; hence $\Delta_D\mathcal N\subset\mathcal N$.

\smallskip
\textit{(iii) Conclusion.} Suppose $\mathcal N\neq\{0\}$. Then the linear map $\Delta_D|_{\mathcal N}$ has an eigenvector: $v_0\in\mathcal N\setminus\{0\}$, $\Delta_Dv_0=\mu v_0$ for some $\mu\in\mathbb C$. We have $\mu\neq1$, since otherwise $(1-\Delta_D)v_0=0$ and $v_0=0$. Thus $v_0=(1-\mu)\mathcal Rv_0$ and, iterating, $v_0\in H^{\sigma+2m}_D(\Omega)$ for every $m\in\mathbb N$; in particular $v_0\in\mathcal D(\Delta_D)$ and $\Delta v_0=\mu v_0$ in $\mathcal D'(\Omega)$. By elliptic regularity, $v_0$ is real analytic in $\Omega$. Taking $t=0$ in the definition of $\mathcal N^\sigma_{T_1}$ we get $\varphi v_0=0$, so $v_0$ vanishes on the nonempty open set $\{x\in\Omega:|x|>R+1\}$, where $\varphi=1$. Since $\Omega$ is connected, $v_0\equiv0$, a contradiction. Therefore $\mathcal N^\sigma_{T_1}=\{0\}$, and $\mathcal N^\sigma_{T_0}\subset\mathcal N^\sigma_{T_1}=\{0\}$ because $T_0\geq T_1$. Finally, $\mathcal N^2_T\subset\mathcal N^0_T=\{0\}$.
\end{proof}

\begin{remark}
The spectral argument of Lemma \ref{UCP} uses that the free equation is invariant under time translations. It replaces the Paley--Wiener--Schwartz argument available on $\mathbb{R}^3$ \cite[Lemma 3.1]{BrCaCaDo}, which relies on the explicit kernel of $e^{it\Delta}$ and has no analogue on $\Omega$, and it does not apply to the linearized equation around a nontrivial trajectory, whose potentials depend on time (see Section \ref{Sec4}).
\end{remark}

\vspace{0.1cm}

\noindent\textbf{Step 4.} \textit{Observability inequality.}

\begin{proposition}\label{obs}
Assume \textup{(H)}, let $T>0$ and $\sigma\in[-2,2]$. There exists $C>0$ such that every solution of \eqref{adj} with $v_0\in H^\sigma_D(\Omega)$ satisfies
\begin{equation}\label{OI_new}
\|v_0\|^2_{H^\sigma_D(\Omega)}\leq C\int_0^T\|\varphi v(t)\|^2_{H^\sigma_D(\Omega)}\,dt .
\end{equation}
\end{proposition}

\begin{proof}
Let first $\sigma\in\mathcal I$. If \eqref{OI_new} fails, there are $v_{j,0}\in H^\sigma_D(\Omega)$ with
\begin{equation}\label{contr}
\|v_{j,0}\|_{H^\sigma_D}=1,\qquad\int_0^T\|\varphi v_j(t)\|^2_{H^\sigma_D}\,dt\leq\frac1j .
\end{equation}
Extracting a subsequence, $v_{j,0}\rightharpoonup v_0$ weakly in $H^\sigma_D(\Omega)$. Arguing as in Lemma \ref{compact}(c), $\varphi v_j\rightharpoonup\varphi v$ weakly in $L^2(0,T;H^\sigma_D(\Omega))$, while $\varphi v_j\to0$ strongly by \eqref{contr}. Hence $\varphi v=0$ a.e. in $(0,T)$ and, by continuity of $t\mapsto\varphi v(t)\in H^\sigma_D(\Omega)$, for every $t\in[0,T]$. Thus $v_0\in\mathcal N^\sigma_T=\{0\}$ by Lemma \ref{UCP}. By Lemma \ref{compact}, $Q_\sigma(v_{j,0})=Q_\sigma(v_{j,0}-v_0)\to0$, and Lemma \ref{relaxed} gives
$$
1=\|v_{j,0}\|^2_{H^\sigma_D}\leq C\left(\frac1j+Q_\sigma(v_{j,0})\right)\longrightarrow0,
$$
a contradiction.

Let now $\sigma=2$ and set $A:=1-\Delta_D$. For $v_0\in H^2_D(\Omega)$, the function $Av(t)=e^{it\Delta_D}Av_0$ solves \eqref{adj} with initial data $Av_0\in L^2(\Omega)$, and $\|v_0\|_{H^2_D}=\|Av_0\|_{L^2}$. Since $$A(\varphi v)-\varphi Av=-[\Delta,\varphi]v,$$ see the proof of Lemma \ref{relaxed}, $\varphi v(t)\in H^2_D(\Omega)$, and $[\Delta,\varphi]$ is bounded from $H^1_D(\Omega)$ to $L^2(\Omega)$, we have
$$
\|\varphi Av(t)\|_{L^2}\leq\|\varphi v(t)\|_{H^2_D}+C\|v(t)\|_{H^1_D}=\|\varphi v(t)\|_{H^2_D}+C\|v_0\|_{H^1_D}.
$$
Applying \eqref{OI_new} with $\sigma=0$ to $Av$, we obtain
$$
\|v_0\|^2_{H^2_D}\leq C_1\int_0^T\|\varphi v(t)\|^2_{H^2_D}\,dt+C_2T\|v_0\|^2_{H^1_D}.
$$
By the spectral theorem, $$\|v_0\|^2_{H^1_D}\leq\varepsilon\|v_0\|^2_{H^2_D}+C_\varepsilon\|v_0\|^2_{L^2},$$ for every $\varepsilon>0$, and, by \eqref{OI_new} with $\sigma=0$ applied to $v$, $$\|v_0\|^2_{L^2}\leq C\int_0^T\|\varphi v(t)\|^2_{L^2}\,dt\leq C\int_0^T\|\varphi v(t)\|^2_{H^2_D}\,dt.$$ Choosing $C_2T\varepsilon=\frac12$ gives \eqref{OI_new} with $\sigma=2$.
\end{proof}

\vspace{0.1cm}

\noindent\textbf{Step 5.} \textit{Exact controllability: Proof of Theorem \ref{lin_th}}

Consider the control system
\begin{equation}\label{lin_new}
\begin{cases}
iz_t+\Delta z=\varphi(x)h(x,t), & x\in\Omega,\ t\in(0,T),\\
z(x,t)=0, & x\in\partial\Omega,\ t\in(0,T),\\
z(x,0)=z_0(x), & x\in\Omega .
\end{cases}
\end{equation}
Let $s\in[-2,2]$. For every $z_0\in H^s_D(\Omega)$ and every $h\in L^1(0,T;H^s_D(\Omega))$, problem \eqref{lin_new} has a unique solution $z\in C([0,T];H^s_D(\Omega))$, given by Duhamel's formula.

\begin{theorem}\label{exact_lin}
Assume \textup{(H)}, let $T>0$ and $s\in[-2,2]$. There exists a bounded linear operator
$$
G:H^s_D(\Omega)\times H^s_D(\Omega)\longrightarrow C([0,T];H^s_D(\Omega))
$$
such that for any $z_0,z_T\in H^s_D(\Omega)$ the solution of \eqref{lin_new} with $h=G(z_0,z_T)$ satisfies $z(T)=z_T$. The control force $\varphi h$ is supported in $\{x\in\Omega:|x|\geq R\}\times[0,T]$.
\end{theorem}

\begin{proof}
Let $\sigma:=-s\in[-2,2]$ and $\Lambda_s:=(1-\Delta_D)^s$, an isometric isomorphism from $H^s_D(\Omega)$ onto $H^{-s}_D(\Omega)=H^\sigma_D(\Omega)$ with $\langle\Lambda_s^{-1}g,g\rangle=\|g\|^2_{H^\sigma_D}$, for $g\in H^\sigma_D(\Omega)$.

\smallskip
\textit{Null controllability.} For $v_0\in H^\sigma_D(\Omega)$ let $v$ solve \eqref{adj}, set
\begin{equation}\label{HUMcontrol}
h(t):=-i\,\Lambda_s^{-1}\big(\varphi v(t)\big)\in C([0,T];H^s_D(\Omega)),
\end{equation}
let $z$ be the solution of \eqref{lin_new} (backward in time) with $z(T)=0$, and define $\mathcal S v_0:=z(0)\in H^s_D(\Omega)$. For smooth data, $$\frac{d}{dt}\langle z,v\rangle=\langle i\Delta z-i\varphi h,v\rangle+\langle z,i\Delta v\rangle=-i\langle\varphi h,v\rangle,$$ so that, by density,
$$
\langle\mathcal Sv_0,v_0\rangle=i\int_0^T\langle h(t),\varphi v(t)\rangle\,dt=\int_0^T\langle\Lambda_s^{-1}(\varphi v),\varphi v\rangle\,dt=\int_0^T\|\varphi v(t)\|^2_{H^\sigma_D}\,dt .
$$
By Proposition \ref{obs}, $\langle\mathcal Sv_0,v_0\rangle\geq C^{-1}\|v_0\|^2_{H^\sigma_D}$. Since $\mathcal S:H^\sigma_D\to H^s_D=(H^\sigma_D)'$ is bounded and linear, the Lax--Milgram theorem shows that $\mathcal S$ is an isomorphism. Given $z_0\in H^s_D(\Omega)$, the choice $v_0=\mathcal S^{-1}z_0$ in \eqref{HUMcontrol} drives $z_0$ to $0$ at time $T$, by uniqueness of solutions of \eqref{lin_new}.

\smallskip
\textit{Exact controllability.} The solution of \eqref{lin_new} is $$z(t)=e^{it\Delta_D}z_0+L_h(t),$$ with $L_h$ linear in $h$. Thus $z(T)=z_T$ if and only if $h$ drives $z_0-e^{-iT\Delta_D}z_T$ to $0$, and we set $$G(z_0,z_T):=-i\Lambda_s^{-1}\big(\varphi\,e^{it\Delta_D}\mathcal S^{-1}(z_0-e^{-iT\Delta_D}z_T)\big).$$ Boundedness of $G$ follows from the boundedness of $\mathcal S^{-1}$, $\Lambda_s^{-1}$, $e^{it\Delta_D}$ and of multiplication by $\varphi$. Since \eqref{lin_new} coincides with \eqref{int_lin}, this also proves Theorem \ref{lin_th}.
\end{proof}

\begin{remarks} Let us give some remarks.
\begin{itemize}
\item[(1)] The forcing term in \eqref{lin_new} is $\varphi h=A v$, with $A:=-i\varphi(1-\Delta_D)^{-s}\varphi$, which is exactly the HUM operator used by Laurent \cite[Proposition 7.2]{Laurent2}. Only $\varphi h$, not $h$, is localized, as in \cite[Remark 1]{BrCaCaDo}.
\item[(2)] The geometric assumption is used only through Proposition \ref{smoothing}; no star-shapedness of $\Theta$ is needed, and no normal trace of $v$ on $\partial\Omega$ appears in the proof of Proposition \ref{obs}.
\item[(3)] For the nonlinear results we only use $s=1$ (control in $H^1_0(\Omega)$) and $s=2$ (control in $H^2_D(\Omega)$), and $s\in(\frac n2,2]$ in Remark \ref{rem_algebra}. The whole range $s\in[-2,2]$ is obtained without any commutator between $\varphi$ and fractional powers of $\Delta_D$: for $\sigma\in(-1,0)$, the integer identity $\varphi\mathcal R-\mathcal R\varphi=-\mathcal R[\Delta,\varphi]\mathcal R$ is combined with the relaxed inequality at the level $\tau=\sigma+2\in(1,2)$, at the price of the weaker remainder in \eqref{Qsigma}. Moreover, the cut-offs are only used as multipliers on $H^\sigma_D(\Omega)$, $\sigma\in[-2,2]$, so that no compatibility condition of higher order on $\partial\Omega$ appears.
\item[(4)] Proposition \ref{smoothing} is only used with $s\in[0,1]$, which is exactly the range proved in \cite{BGT1}; for $\sigma\in[-2,0)$ it is applied to $y_0=\mathcal Rv_0\in H^{\sigma+2}_D(\Omega)$.
\end{itemize}
\end{remarks}

\subsection{Nonlinear results}
This subsection is devoted to the proofs of Theorems \ref{nlin_th} and \ref{principal} and of Corollary \ref{quintic}. Let $m\in\{3,5\}$ and $g(v):=\lambda|v|^{m-1}v$. Note that the case $m=3$ corresponds to \eqref{s1a}, and the case $m=5$ to Corollary \ref{quintic}. We consider
\begin{equation}\label{nlinear}
\begin{cases}
iv_t+\Delta v+g(v)=\mathcal{K}h,&\text{ in }\Omega\times(0,T),\\
v(x,t)=0,&\text{ on }\partial\Omega\times(0,T),\\
v(x,0)=v_0(x),&\text{ in }\Omega,
\end{cases}
\end{equation}
where
\begin{equation}\label{KH}
(\mathcal{K}h)(x,t):=\varphi(x)h(x,t),
\end{equation}
and $\varphi$ is the cut-off in \eqref{phi}, so that $\mathcal Kh$ is supported in $\overline{\omega}\times[0,T]$, see \eqref{omega}. Throughout this subsection assumption \textup{(H)} of Subsection \ref{sec:linear} is in force, and we write $W(t):=e^{it\Delta_D}$. Since $$iv_t=-\Delta v-g(v)+\mathcal Kh,$$ a solution of \eqref{nlinear} is a fixed point of
\begin{equation}\label{duhamel}
v(t)=W(t)v_0+i\int_0^tW(t-r)g(v)(r)\,dr-i\int_0^tW(t-r)(\mathcal Kh)(r)\,dr .
\end{equation}

\subsubsection{An abstract perturbation result}

\begin{theorem}\label{control_n}
Assume \textup{(H)}. Let $T>0$, $s\in[1,2]$, $m\in\{3,5\}$, and let $X_{T}$ be a Banach space with $X_T\hookrightarrow C([0,T];H^s_D(\Omega))$, that is, $\|u\|_{C([0,T];H^s_D)}\leq\|u\|_{X_T}$. Assume that there exists $C_0\geq1$ such that
\begin{itemize}
\item[(i)] $\|W(\cdot)u_0\|_{X_T}\leq C_0\|u_0\|_{H^s_D(\Omega)}$ for all $u_0\in H^s_D(\Omega)$;
\item[(ii)] $\big\|\displaystyle\int_0^tW(t-r)f(r)\,dr\big\|_{X_T}\leq C_0\|f\|_{L^1(0,T;H^s_D(\Omega))}$, \quad for all $f\in L^1(0,T;H^s_D(\Omega))$;
\item[(iii)] for all $u,v\in X_T$, $g(v)\in L^1(0,T;H^s_D(\Omega))$ and
$$
\|g(v)\|_{L^1(0,T;H^s_D)}\leq C_0\|v\|^m_{X_T},\qquad \|g(v)-g(u)\|_{L^1(0,T;H^s_D)}\leq C_0\big(\|v\|^{m-1}_{X_T}+\|u\|^{m-1}_{X_T}\big)\|v-u\|_{X_T}.
$$
\end{itemize}
Then there exists $\delta>0$ such that for any $v_0,v_T\in H^s_D(\Omega)$ with $\|v_0\|_{H^s_D}\leq\delta$ and $\|v_T\|_{H^s_D}\leq\delta$, there is a control $h\in C([0,T];H^s_D(\Omega))$ such that \eqref{nlinear} admits a solution $v\in X_T$ with $v(0)=v_0$ and $v(T)=v_T$.
\end{theorem}

\begin{proof}
Let $G$ be the operator given by Theorem \ref{exact_lin}: for $z_0,z_T\in H^s_D(\Omega)$, the solution of \eqref{lin_new} with $h=G(z_0,z_T)$ starts at $z_0$ and satisfies $z(T)=z_T$, and $$\|G(z_0,z_T)\|_{C([0,T];H^s_D)}\leq C_G(\|z_0\|_{H^s_D}+\|z_T\|_{H^s_D}).$$ Recall also that multiplication by $\varphi$ is bounded on $H^s_D(\Omega)$, with norm $C_\varphi$.

For $v\in X_T$ set
$$
N(v)(t):=i\int_0^tW(t-r)g(v)(r)\,dr,\qquad h_v:=G\big(v_0,\,v_T-N(v)(T)\big),
$$
and
$$
\mathcal T(v)(t):=W(t)v_0+N(v)(t)-i\int_0^tW(t-r)(\varphi h_v)(r)\,dr .
$$
By definition of $G$, the linear part $$W(t)v_0-i\int_0^tW(t-r)\varphi h_v\,dr$$ equals $v_0$ at $t=0$ and $v_T-N(v)(T)$ at $t=T$. Hence $\mathcal T(v)(0)=v_0$ and $\mathcal T(v)(T)=v_T$, and by \eqref{duhamel} a fixed point of $\mathcal T$ is a solution of \eqref{nlinear} with control $h=h_v$ and the required properties.

By (ii)--(iii), $$\|N(v)\|_{X_T}\leq C_0^2\|v\|^m_{X_T}$$ and, in particular, $$\|N(v)(T)\|_{H^s_D}\leq C_0^2\|v\|^m_{X_T}.$$ Therefore, using (i), (ii) and $$\|\varphi h_v\|_{L^1(0,T;H^s_D)}\leq TC_\varphi\|h_v\|_{C([0,T];H^s_D)},$$ we get
$$
\|\mathcal T(v)\|_{X_T}\leq C_1\big(\|v_0\|_{H^s_D}+\|v_T\|_{H^s_D}\big)+C_2\|v\|^m_{X_T},
$$
with $$C_1:=C_0+C_0TC_\varphi C_G$$ and $$C_2:=C_0^2(1+C_0TC_\varphi C_G).$$ Similarly, since $h_v-h_u=-G(0,N(v)(T)-N(u)(T))$,
$$
\|\mathcal T(v)-\mathcal T(u)\|_{X_T}\leq C_2\big(\|v\|^{m-1}_{X_T}+\|u\|^{m-1}_{X_T}\big)\|v-u\|_{X_T}.
$$
Let $M:=4C_1\delta$ and choose $\delta>0$ so small that $4C_2M^{m-1}\leq1$. If $\|v_0\|_{H^s_D},\|v_T\|_{H^s_D}\leq\delta$, then for $u,v$ in the closed ball $B_M\subset X_T$ we get $$\|\mathcal T(v)\|_{X_T}\leq\frac M2+C_2M^m\leq M$$ and $$\|\mathcal T(v)-\mathcal T(u)\|_{X_T}\leq2C_2M^{m-1}\|v-u\|_{X_T}\leq\frac12\|v-u\|_{X_T}.$$ The contraction mapping principle gives the fixed point.
\end{proof}

\subsubsection{Verification of the assumptions}

We first record the Strichartz-type estimate we need in the energy space. In dimension two, it is a consequence of the Strichartz estimates with loss of \cite[Proposition 2.11]{BGT1} (see also \cite[Corollary 1.3]{Anton}); in dimension three it is the estimate of Blair, Smith and Sogge recalled in Proposition \ref{BSS_est}.

\begin{lemma}\label{L4Linfty}
Assume \textup{(H)} with $n\in\{2,3\}$, and let $p=4$ if $n=3$ and $p\in(2,\infty)$ if $n=2$. For every $T>0$ there exists $C>0$ such that
\begin{equation}\label{LpLinfty}
\|W(\cdot)u_0\|_{L^p(0,T;L^\infty(\Omega))}\leq C\|u_0\|_{H^1_0(\Omega)},\qquad\forall\,u_0\in H^1_0(\Omega).
\end{equation}
\end{lemma}

\begin{proof}
If $n=3$, this is Proposition \ref{BSS_est}. If $n=2$, by \cite[Proposition 2.11]{BGT1}, for $s'\in[0,1]$ and $(p,q)$ with $p>2$ and $\frac2p+\frac2q=1$,
$$
\|W(\cdot)u_0\|_{L^p(0,T;W^{s',q}(\Omega))}\leq C\|u_0\|_{H^{s'+1/p}_D(\Omega)} .
$$
Take $s'=1-\frac1p\in[0,1]$, so that the right-hand side is $C\|u_0\|_{H^1_0}$, and $q=\frac{2p}{p-2}$. Since $$s'q=\frac{2(p-1)}{p-2}>2=n$$ and $\Omega$ has compact smooth boundary, the Sobolev embedding $W^{s',q}(\Omega)\hookrightarrow L^\infty(\Omega)$ holds, which gives \eqref{LpLinfty}.
\end{proof}

\begin{proposition}[$H^1_0$--controllability, $n=2,3$]\label{H1-Control}
Assume \textup{(H)} with $n\in\{2,3\}$, and let $T>0$ and $m\in\{3,5\}$. The assumptions of Theorem \ref{control_n} hold with $s=1$ and
$$
X_T:=C([0,T];H^1_0(\Omega))\cap L^4(0,T;L^\infty(\Omega)),\qquad \|u\|_{X_T}:=\|u\|_{L^\infty(0,T;H^1_0)}+\|u\|_{L^4(0,T;L^\infty)} .
$$
Consequently, system \eqref{nlinear} is locally exactly controllable in $H^1_0(\Omega)$.
\end{proposition}

\begin{proof}
(i) follows from the unitarity of $W(t)$ on $H^1_0(\Omega)$ and Lemma \ref{L4Linfty} with $p=4$.

(ii) For the $C([0,T];H^1_0)$ part we use unitarity. For the $L^4L^\infty$ part, by Minkowski's inequality and Lemma \ref{L4Linfty} applied on the interval $[0,T-r]$ (with a constant uniform in $r\in[0,T]$, since $[0,T-r]\subset[0,T]$),
$$
\Big\|\int_0^t W(t-r)f(r)\,dr\Big\|_{L^4(0,T;L^\infty)}\leq\int_0^T\big\|W(\cdot-r)f(r)\big\|_{L^4(r,T;L^\infty)}\,dr\leq C\|f\|_{L^1(0,T;H^1_0)} .
$$

(iii) If $v\in H^1_0(\Omega)\cap L^\infty(\Omega)$, then $g(v)=\lambda v^{\frac{m+1}2}\overline v^{\frac{m-1}2}\in H^1_0(\Omega)$,
$$
\nabla g(v)=\lambda\Big(\tfrac{m+1}2|v|^{m-1}\nabla v+\tfrac{m-1}2|v|^{m-3}v^2\nabla\overline v\Big)\quad\text{and}\quad \|g(v)\|_{H^1_0}\leq m|\lambda|\,\|v\|^{m-1}_{L^\infty}\|v\|_{H^1_0}.
$$
By H\"older's inequality in time, since $m-1\leq4$,
$$
\|g(v)\|_{L^1(0,T;H^1_0)}\leq m|\lambda|\,\|v\|_{L^\infty(0,T;H^1_0)}\int_0^T\|v(t)\|^{m-1}_{L^\infty}dt\leq m|\lambda|T^{\frac{5-m}4}\|v\|_{X_T}^m .
$$
For the difference, we use the pointwise bounds
$$
\begin{gathered}
|g(v)-g(u)|\leq C(|v|^{m-1}+|u|^{m-1})|v-u|,\\
|\nabla(g(v)-g(u))|\leq C(|v|^{m-1}+|u|^{m-1})|\nabla(v-u)|+C(|v|^{m-2}+|u|^{m-2})(|\nabla v|+|\nabla u|)|v-u|,
\end{gathered}
$$
which give, for a.e. $t$,
$$
\begin{aligned}
\|g(v)-g(u)\|_{H^1_0}\leq{}& C\big(\|v\|^{m-1}_{L^\infty}+\|u\|^{m-1}_{L^\infty}\big)\|v-u\|_{H^1_0}\\
&+C\big(\|v\|^{m-2}_{L^\infty}+\|u\|^{m-2}_{L^\infty}\big)\big(\|v\|_{H^1_0}+\|u\|_{H^1_0}\big)\|v-u\|_{L^\infty}.
\end{aligned}
$$
Integrating in time and using H\"older's inequality (with exponents $\frac{m-2}4+\frac14+\frac{5-m}4=1$ in the last term),
$$
\|g(v)-g(u)\|_{L^1(0,T;H^1_0)}\leq CT^{\frac{5-m}4}\big(\|v\|^{m-1}_{X_T}+\|u\|^{m-1}_{X_T}\big)\|v-u\|_{X_T}.
$$
Theorem \ref{control_n} applies.
\end{proof}

\begin{proposition}[$H^2_D$--controllability, $n=2,3$]\label{H2-Control}
Assume \textup{(H)} with $n\in\{2,3\}$, and let $T>0$ and $m\in\{3,5\}$. The assumptions of Theorem \ref{control_n} hold with $s=2$ and $X_T:=C([0,T];H^2_D(\Omega))$, where $H^2_D(\Omega)=H^2(\Omega)\cap H^1_0(\Omega)$. Consequently, system \eqref{nlinear} is locally exactly controllable in $H^2(\Omega)\cap H^1_0(\Omega)$.
\end{proposition}

\begin{proof}
(i) and (ii) follow from the unitarity of $W(t)$ on $H^2_D(\Omega)$ and Minkowski's inequality. Since $n\leq3$ and $\Omega$ has compact smooth boundary, $H^2(\Omega)$ is an algebra contained in $L^\infty(\Omega)$; moreover, $g(v)=\lambda v^{\frac{m+1}2}\overline v^{\frac{m-1}2}$ vanishes on $\partial\Omega$ when $v$ does. Hence $g(v)\in H^2_D(\Omega)$ with
$$
\|g(v)\|_{H^2_D}\leq C\|v\|^m_{H^2_D},\qquad\|g(v)-g(u)\|_{H^2_D}\leq C\big(\|v\|^{m-1}_{H^2_D}+\|u\|^{m-1}_{H^2_D}\big)\|v-u\|_{H^2_D},
$$
and (iii) follows after integration in time, with a factor $T$.
\end{proof}

\begin{remark}\label{rem_algebra}
Since $H^s_D(\Omega)=H^s(\Omega)\cap H^1_0(\Omega)$ for $s\in[1,2]$, and $H^s(\Omega)$ is an algebra contained in $L^\infty(\Omega)$ for $s>\frac n2$, the proof of Proposition \ref{H2-Control} applies verbatim to $X_T:=C([0,T];H^s_D(\Omega))$ for every $s\in(\frac n2,2]$. Since Theorem \ref{exact_lin} provides the linear control operator for every $s\in[-2,2]$, and Lemma \ref{extension} below holds for every $s\in[1,2]$, Theorems \ref{principal} and \ref{nlin_th} and Corollary \ref{quintic} remain valid for these values of $s$.
\end{remark}

\begin{proof}[Proof of Theorem \ref{nlin_th}]
It follows from Theorem \ref{control_n} with $m=3$, combined with Proposition \ref{H1-Control} (case $s=1$) and Proposition \ref{H2-Control} (case $s=2$).
\end{proof}

Finally, we present the proof of Theorem \ref{principal}. We need an extension operator which preserves the Dirichlet condition on $\Gamma_2=\partial\Theta$.

\begin{lemma}\label{extension}
Let $s\in[1,2]$. There exists a bounded linear operator $E:X^s(\Omega_0)\to H^s_D(\Omega)$ such that $(Eu)|_{\Omega_0}=u$ for every $u\in X^s(\Omega_0)$.
\end{lemma}

\begin{proof}
Choose $\varepsilon>0$ with $\Theta\subset B_{R_0-3\varepsilon}(0)$ and $\zeta,\tilde\zeta\in C^\infty(\mathbb{R}^n;[0,1])$ with $\zeta=0$ on $B_{R_0-2\varepsilon}(0)$, $\zeta=1$ for $|x|\geq R_0-\varepsilon$, $\tilde\zeta=0$ on $B_{R_0-3\varepsilon}(0)$ and $\tilde\zeta=1$ on $\operatorname{supp}\zeta$. Let $$\mathcal A:=\{R_0-3\varepsilon<|x|<R_0\}$$ and let $$E_{\mathcal A}:H^s(\mathcal A)\to H^s(\mathbb{R}^n)$$ be a standard bounded extension operator. For $u\in X^s(\Omega_0)$ define
$$
Eu:=\widetilde{(1-\zeta)u}+\tilde\zeta\,E_{\mathcal A}\big((\zeta u)|_{\mathcal A}\big),
$$
where $\widetilde{(1-\zeta)u}$ denotes the extension by zero of $(1-\zeta)u$ to $\Omega$. Since $(1-\zeta)u$ vanishes for $|x|\geq R_0-\varepsilon$, its extension by zero belongs to $H^s(\Omega)$; since $\tilde\zeta$ vanishes near $\Theta$, the second term belongs to $H^s(\Omega)$ and vanishes near $\partial\Omega$. Moreover $Eu=0$ on $\partial\Omega=\Gamma_2$ because $u=0$ there. Hence $$Eu\in H^s(\Omega)\cap H^1_0(\Omega)=H^s_D(\Omega),$$ and $$Eu=(1-\zeta)u+\zeta u=u$$ in $\Omega_0$, because $\tilde\zeta=1$ on $\operatorname{supp}\zeta$ and $E_{\mathcal A}(\zeta u)=\zeta u$ in $\mathcal A$.
\end{proof}

\begin{proof}[Proof of Theorem \ref{principal}]
We apply Theorem \ref{nlin_th} with the cut-off $\varphi$ in \eqref{phi} associated with $R:=R_0+1$, so that $\varphi=0$ on $B_{R_0+1}(0)\supset\overline{\Omega_0}$. Let $\delta_1>0$ be given by Theorem \ref{nlin_th} and set $\delta:=\delta_1/\|E\|$, where $E$ is the operator of Lemma \ref{extension}. For $u_0,u_1\in X^s(\Omega_0)$ with norms not larger than $\delta$, we have $\|Eu_0\|_{H^s_D},\|Eu_1\|_{H^s_D}\leq\delta_1$, and Theorem \ref{nlin_th} provides $\tilde h\in C([0,T];H^s_D(\Omega))$ and a solution $v\in C([0,T];H^s_D(\Omega))$ of
$$
\begin{cases}
iv_t+\Delta v+\lambda|v|^2v=\varphi\tilde h,&\ \text{ in }\Omega\times(0,T),\\ v=0,&\ \text{ on }\partial\Omega\times(0,T),
\end{cases}
$$
with $v(0)=Eu_0$ and $v(T)=Eu_1$. Set
$$
u:=v|_{\Omega_0\times[0,T]},\qquad\text{and}\qquad h:=v|_{\Gamma_1\times[0,T]} .
$$
Since $\varphi=0$ on $\Omega_0$, $u$ solves \eqref{s1} in $\mathcal{D}'(\Omega_0\times(0,T))$, and $u\in C([0,T];H^s(\Omega_0))$ with $u(0)=u_0$, $u(T)=u_1$ by Lemma \ref{extension}. Since $v(t)\in H^1_0(\Omega)$, $u=0$ on $\Gamma_2$. Finally, $\Gamma_1$ is a smooth hypersurface contained in the interior of $\Omega$, so the trace theorem gives $h\in C([0,T];H^{s-1/2}(\Gamma_1))$ and $u=h$ on $\Gamma_1$. Thus $u$ solves \eqref{s1}--\eqref{s2} and satisfies \eqref{1.1}.
\end{proof}

\begin{proof}[Proof of Corollary \ref{quintic}]
The internal controllability follows from Theorem \ref{control_n} with $m=5$, combined with Propositions \ref{H1-Control} and \ref{H2-Control}, which cover this case. The boundary controllability is then deduced exactly as in the proof of Theorem \ref{principal}, since the extension and restriction arguments do not depend on the nonlinearity.
\end{proof}

\section{Internal controllability around a nontrivial trajectory \label{Sec4}}
Throughout this section we assume \textup{(H)} of Subsection \ref{sec:linear}, \textup{(W)} and, when stated, \textup{(UCP$_w$)}. We keep the notation $\varphi,\chi,\psi,\Omega',\mathcal R$ of Subsection \ref{sec:linear}, and set $V_1=2\lambda|w|^2$, $V_2=\lambda w^2$. Since $V_1,V_2\in C([0,T];W^{2,\infty}(\Omega))$, multiplication by $V_j$ (and complex conjugation) defines real-linear operators which are bounded on $H^\theta_D(\Omega)$ for every $\theta\in[-2,2]$, uniformly in $t\in[0,T]$ (for $\theta\in[0,2]$ directly, for $\theta\in[-2,0]$ by duality). Consequently \eqref{linearized_int} with $h=0$ and \eqref{adj_w} are well posed, forward and backward, in every $H^\theta_D(\Omega)$, $\theta\in[-2,2]$, by Duhamel's formula and Gronwall's lemma. We denote by $U(t)$ the (real-linear) propagator of \eqref{adj_w}, so that $\|U(t)v_0\|_{H^\theta_D}$ and $\|v_0\|_{H^\theta_D}$ are equivalent, uniformly in $t\in[0,T]$.

\subsection{Linear result}

We first extend the local smoothing effect to negative regularities and to \eqref{adj_w}.

\begin{lemma}\label{smoothing_neg}
Assume \textup{(H)}. For every $T>0$, every real-valued $\eta\in C^\infty_0(\mathbb{R}^n)$ and every $s\in[-2,1]$ there exists $C>0$ such that
\begin{equation}\label{inhom_smooth}
\Big\|\eta\int_0^tW(t-r)\,\eta f(r)\,dr\Big\|_{L^2(0,T;H^{s+1}_D)}\leq C\|\eta f\|_{L^2(0,T;H^{s}_D)} .
\end{equation}
\end{lemma}

\begin{proof}
For $s\in[-1,1]$ this is \cite[Proposition 2.7, (2.14)]{BGT1}. Let $\eta_1\in C^\infty_0(\mathbb{R}^n)$ be real-valued and equal to $1$ on $\operatorname{supp}\eta$. The proof of \cite[(2.14)]{BGT1} only uses Plancherel's theorem in time and uniform bounds $$\|\eta_1(-\Delta_D-z)^{-1}\eta_1\|_{H^s_D\to H^{s+1}_D}\leq C,$$ for the spectral parameters $z$ which occur, namely $z=(\kappa\pm i\varepsilon)^2$ with $\kappa\in\mathbb{R}$ and $0<\varepsilon\ll1$ (see \eqref{l9}), and $\operatorname{Re}z<0$ \cite[Remark 2.6]{BGT1}; indeed, $$\eta(-\Delta_D-z)^{-1}(\eta f)=\eta\,[\eta_1(-\Delta_D-z)^{-1}\eta_1](\eta f),$$ which explains why only $\|\eta f\|$ appears on the right-hand side of \eqref{inhom_smooth}. Since $\eta_1$ is real-valued, the adjoint of $\eta_1(-\Delta_D-z)^{-1}\eta_1$ with respect to $\langle\cdot,\cdot\rangle$ is $\eta_1(-\Delta_D-\overline z)^{-1}\eta_1$, and this set of parameters is invariant under conjugation. Hence the bound with $s=1$ gives the same bound from $H^{-2}_D$ to $H^{-1}_D$ and, by interpolation with the case $s=-1$, for every $s\in[-2,-1]$. The proof of \cite[(2.14)]{BGT1} then applies verbatim.
\end{proof}

For $\sigma\in\{-1,-2\}$ we set $\tau:=\sigma+2\in\{1,0\}$ and, for $v_0\in H^\sigma_D(\Omega)$,
$$
v(t):=U(t)v_0,\qquad y(t):=\mathcal Rv(t),\qquad y_0:=\mathcal Rv_0 ,
$$
so that $\|y_0\|_{H^\tau_D}=\|v_0\|_{H^\sigma_D}$. Since $$iv_t+\Delta_Dv=-V_1v+V_2\overline v$$ and $\mathcal R$ commutes with $W(t)$,
\begin{equation}\label{y_duhamel}
y(t)=W(t)y_0+i\int_0^tW(t-r)f(r)\,dr,\qquad \text{with} \quad f:=\mathcal R\big(V_1v-V_2\overline v\big),
\end{equation} 
and, for every $\theta\in[0,1]$ such that $v_0\in H^{\theta-2}_D(\Omega)$ (equivalently, $y_0\in H^\theta_D(\Omega)$) and every $r\in[0,T]$,
\begin{equation}\label{f_bound}
\|f(r)\|_{H^\theta_D}\leq C\|v(r)\|_{H^{\theta-2}_D}\leq C\|v_0\|_{H^{\theta-2}_D}=C\|y_0\|_{H^\theta_D},\qquad \|y(r)\|_{H^\theta_D}\leq C\|y_0\|_{H^\theta_D}.
\end{equation}

\begin{lemma}\label{smooth_w}
Assume \textup{(H)} and \textup{(W)}. For every real-valued $\eta\in C^\infty_0(\mathbb{R}^n)$ and every $\theta\in[0,1]$ there exists $C>0$ such that, for every $v_0\in H^{\theta-2}_D(\Omega)$, the function $y(t)=\mathcal RU(t)v_0$ satisfies
$$
\int_0^T\|\eta y(t)\|^2_{H^{\theta+1/2}_D}\,dt\leq C\|y_0\|^2_{H^\theta_D},\qquad y_0=\mathcal Rv_0 .
$$
\end{lemma}

\begin{proof}
Apply Proposition \ref{smoothing} to $W(t)y_0$ and, for the Duhamel term in \eqref{y_duhamel}, Minkowski's inequality together with Proposition \ref{smoothing} on each interval $[r,T]$ (the constant does not depend on $r$, since $[0,T-r]\subset[0,T]$), and \eqref{f_bound}.
\end{proof}

\begin{lemma}[Observability up to a compact term]\label{relaxed_w}
Assume \textup{(H)} and \textup{(W)}, and let $\sigma\in\{-1,-2\}$. There exists $C>0$ such that for every $v_0\in H^\sigma_D(\Omega)$
\begin{equation}\label{relaxed_w_est}
\|v_0\|^2_{H^\sigma_D}\leq C\left(\int_0^T\|\varphi v(t)\|^2_{H^\sigma_D}\,dt+\int_0^T\|\psi y(t)\|^2_{L^2}\,dt\right).
\end{equation}
\end{lemma}

\begin{proof}
By \eqref{f_bound} with $\theta=\tau$ (and the analogous bound backwards in time), $$\|y_0\|^2_{H^\tau_D}\leq\frac CT\int_0^T\|y(t)\|^2_{H^\tau_D}dt\leq\frac{2C}T\int_0^T\big(\|\varphi y\|^2_{H^\tau_D}+\|\chi y\|^2_{H^\tau_D}\big)\,dt.$$ If $\tau=0$ we use $\|\chi y\|_{L^2}\leq\|\psi y\|_{L^2}$. If $\tau=1$, Lemma \ref{smooth_w} with $\eta=\chi$, $\theta=\frac12$ and interpolation give $$\int_0^T\|\chi y\|^2_{H^1_0}\leq C\|y_0\|^2_{H^{1/2}_D}\leq\varepsilon\|y_0\|^2_{H^1_0}+C_\varepsilon\|y_0\|^2_{L^2},$$ while, as above, $$\|y_0\|^2_{L^2}\leq\frac{2C}T\int_0^T(\|\varphi y\|^2_{L^2}+\|\psi y\|^2_{L^2})dt.$$ Absorbing the $\varepsilon$-term, in both cases
$$
\|y_0\|^2_{H^\tau_D}\leq C\int_0^T\big(\|\varphi y\|^2_{H^\tau_D}+\|\psi y\|^2_{L^2}\big)\,dt .
$$
Finally, exactly as in the proof of Lemma \ref{relaxed}, we have $$\|\varphi y\|_{H^\tau_D}=\|\varphi\mathcal Rv\|_{H^\tau_D}\leq\|\varphi v\|_{H^\sigma_D}+C\|\psi y\|_{L^2},$$ and $$\|y_0\|_{H^\tau_D}=\|v_0\|_{H^\sigma_D},$$
giving the result.
\end{proof}

\begin{lemma}[Compactness]\label{compact_w}
Assume \textup{(H)} and \textup{(W)}, and let $\sigma\in\{-1,-2\}$. If $v_{j,0}\rightharpoonup v_0$ weakly in $H^\sigma_D(\Omega)$, then, with $y_j:=\mathcal RU(\cdot)v_{j,0}$ and $y:=\mathcal RU(\cdot)v_0$, $$\int_0^T\|\psi(y_j-y)\|^2_{L^2}dt\to0.$$
\end{lemma}

\begin{proof}
Let $d_j:=y_j-y=\mathcal RU(\cdot)(v_{j,0}-v_0)$. By \eqref{f_bound} and Lemma \ref{smooth_w} with $\theta=0$, $(\psi d_j)$ is bounded in $L^2(0,T;H^{1/2}(\Omega'))$, and by \eqref{y_duhamel}, $\partial_t(\psi d_j)$ is bounded in $L^\infty(0,T;H^{-2}(\Omega'))$ as in Lemma \ref{compact}. By the Aubin--Lions--Simon lemma, $(\psi d_j)$ is relatively compact in $L^2(0,T;L^2(\Omega))$. For every $t$, $U(t)$ and $\mathcal R$ are bounded real-linear maps, hence $\psi d_j(t)\rightharpoonup0$ weakly in $L^2(\Omega)$; since $\|\psi d_j(t)\|_{L^2}$ is bounded uniformly in $j$ and $t$, dominated convergence gives $\psi d_j\rightharpoonup0$ weakly in $L^2(0,T;L^2(\Omega))$. Hence $\psi d_j\to0$ strongly.
\end{proof}

\begin{lemma}[Regularity of invisible solutions]\label{regularity_w}
Assume \textup{(H)} and \textup{(W)}, let $\sigma\in\{-1,-2\}$ and let $v_0\in H^\sigma_D(\Omega)$ be such that $\varphi U(t)v_0=0$ for all $t\in[0,T]$. Then $v_0\in H^1_0(\Omega)$ and $v=U(\cdot)v_0\in C([0,T];H^1_0(\Omega))$.
\end{lemma}

\begin{proof}
Since $\psi-1$ vanishes on $B_{R+2}(0)$ and $\varphi=1$ outside $B_{R+1}(0)$, we have $$(\psi-1)v=(\psi-1)\varphi v=0,$$ i.e. $v=\psi v$. Let $\vartheta\in C^\infty_0(0,T)$, $\vartheta\not\equiv0$, and $u:=\vartheta v$. Then $u(0)=0$, $u=\psi u$ and
$$
iu_t+\Delta_Du=\Phi:=\vartheta\big(-V_1v+V_2\overline v\big)+i\vartheta'v,\qquad \Phi=\psi \Phi,
$$
so that $$u=-i\int_0^tW(t-r)\psi \Phi(r)\,dr$$ and $u=\psi u$. If $v\in C([0,T];H^\theta_D)$ with $\theta\in[-2,0]$, then $\Phi\in L^2(0,T;H^\theta_D)$ and Lemma \ref{smoothing_neg} gives $u\in L^2(0,T;H^{\theta+1}_D)$. Since $u(t)\in H^{\theta+1}_D$ for a.e. $t$ and $\{\vartheta\neq0\}$ has positive measure, we may choose $t_0$ with $\vartheta(t_0)\neq0$ and $u(t_0)\in H^{\theta+1}_D$; then $v(t_0)=u(t_0)/\vartheta(t_0)\in H^{\theta+1}_D$, and the well-posedness of \eqref{adj_w} in $H^{\theta+1}_D$, forward and backward, together with uniqueness in $H^\theta_D$, gives $v\in C([0,T];H^{\theta+1}_D)$. Starting from $\theta=\sigma$ and iterating at most three times we reach $v\in C([0,T];H^1_0(\Omega))$.
\end{proof}

\begin{proposition}[Observability]\label{obs_w}
Assume \textup{(H)}, \textup{(W)} and \textup{(UCP$_w$)}, and let $\sigma\in\{-1,-2\}$. There exists $C>0$ such that every solution of \eqref{adj_w} with $v_0\in H^\sigma_D(\Omega)$ satisfies
\begin{equation}\label{obs_w_est}
\|v_0\|^2_{H^\sigma_D}\leq C\int_0^T\|\varphi v(t)\|^2_{H^\sigma_D}\,dt .
\end{equation}
\end{proposition}

\begin{proof}
Otherwise there are $v_{j,0}$ with $$\|v_{j,0}\|_{H^\sigma_D}=1$$ and $$\int_0^T\|\varphi v_j\|^2_{H^\sigma_D}dt\to0.$$ Up to a subsequence, $v_{j,0}\rightharpoonup v_0$ weakly in $H^\sigma_D$; as in Lemma \ref{compact_w}, $\varphi v_j\rightharpoonup\varphi v$ weakly in $L^2(0,T;H^\sigma_D)$, hence $\varphi v=0$ on $[0,T]$. By Lemma \ref{regularity_w}, $v\in C([0,T];H^1_0(\Omega))$, and, since $\varphi>0$ in $\omega$ by \eqref{phi}, $v=0$ in $\omega\times(0,T)$; by \textup{(UCP$_w$)}, $v\equiv0$, so $v_0=0$ and $y\equiv0$. Lemma \ref{compact_w} gives $$\int_0^T\|\psi y_j\|^2_{L^2}\to0,$$ and Lemma \ref{relaxed_w} yields $$1=\|v_{j,0}\|^2_{H^\sigma_D}\to0,$$ a contradiction.
\end{proof}

\begin{proof}[Proof of Theorem \ref{lin_ar}]
Let $\sigma=-s$ and $\Lambda_s=(1-\Delta_D)^s$. We use the real duality $\operatorname{Re}\langle\cdot,\cdot\rangle$. If $z$ solves \eqref{linearized_int} and $v$ solves \eqref{adj_w}, then for smooth data
$$
\frac{d}{dt}\operatorname{Re}\langle z,v\rangle=\operatorname{Re}\langle -i\varphi h,v\rangle ,
$$
because the contributions of $\Delta$, $V_1$ and $V_2$ cancel (this is the reason for the sign of $V_2$ in \eqref{adj_w}). For $v_0\in H^\sigma_D$ let $h:=-i\Lambda_s^{-1}(\varphi v)\in C([0,T];H^s_D)$, let $z$ solve \eqref{linearized_int} backwards with $z(T)=0$, and set $\mathcal S_wv_0:=z(0)$. Then
$$
\operatorname{Re}\langle\mathcal S_wv_0,v_0\rangle=\operatorname{Re}\Big(i\int_0^T\langle h,\varphi v\rangle\,dt\Big)=\int_0^T\|\varphi v(t)\|^2_{H^\sigma_D}\,dt\geq C^{-1}\|v_0\|^2_{H^\sigma_D}
$$
by Proposition \ref{obs_w}. The Lax--Milgram theorem, applied in the real Hilbert spaces $H^\sigma_D$ and $H^s_D=(H^\sigma_D)'$, shows that the real-linear operator $\mathcal S_w$ is an isomorphism, which gives null controllability. Exact controllability follows as in Theorem \ref{exact_lin}, using the invertibility of the propagator of \eqref{linearized_int} with $h=0$. We denote by $G_w:H^s_D\times H^s_D\to C([0,T];H^s_D)$ the resulting bounded real-linear control operator.
\end{proof}

\subsection{Nonlinear result}

\begin{proof}[Proof of Theorem \ref{nlin_ar}]
Let $X_T$ be the space of Proposition \ref{H1-Control} if $s=1$, and that of Proposition \ref{H2-Control} if $s=2$. Writing $v=w+z$, $z$ solves
\begin{equation}\label{z_eq}
iz_t+\Delta z+V_1z+V_2\overline z+F(w,z)=\varphi h,\qquad\text{with}\quad F(w,z):=\lambda\big(2w|z|^2+\overline wz^2+|z|^2z\big),
\end{equation}
and  $z(0)=z_0:=v_0-w(0)$, we must reach $z(T)=z_T:=v_T-w(T)$. Let $S(t,r)$ be the propagator of \eqref{linearized_int} with $h=0$. Because of the term $V_2\overline z$, $S(t,r)$ is only real-linear, so that in general $S(t,r)(if)\neq iS(t,r)f$; the factor $i$ must therefore be kept inside the propagator, and \eqref{z_eq} is equivalent to
$$
z(t)=S(t,0)z_0+\int_0^tS(t,r)\big(iF(w,z)(r)\big)\,dr+\int_0^tS(t,r)\big(-i\varphi h(r)\big)\,dr .
$$
In the same way, the solution of \eqref{linearized_int} is $$S(t,0)z_0+\int_0^tS(t,r)(-i\varphi h(r))\,dr,$$ and $G_w$ is defined accordingly. For $z\in X_T$ we set
$$
N_w(z)(t):=\int_0^tS(t,r)\big(iF(w,z)(r)\big)\,dr,\qquad h_z:=G_w\big(z_0,\,z_T-N_w(z)(T)\big),
$$
and
$$
\mathcal T_w(z)(t):=S(t,0)z_0+N_w(z)(t)+\int_0^tS(t,r)\big(-i\varphi h_z(r)\big)\,dr .
$$
By definition of $G_w$, the sum of the first and last terms equals $z_0$ at $t=0$ and $z_T-N_w(z)(T)$ at $t=T$, so that $\mathcal T_w(z)(0)=z_0$ and $\mathcal T_w(z)(T)=z_T$, and a fixed point of $\mathcal T_w$ solves \eqref{z_eq} with $h=h_z$. Since $$h_z-h_{\tilde z}=-G_w\big(0,N_w(z)(T)-N_w(\tilde z)(T)\big)$$ and multiplication by $i$ preserves all the norms involved, the estimates of the proof of Theorem \ref{control_n} carry over, with $W(t-r)$ replaced by $S(t,r)$, $g$ by $F(w,\cdot)$ and $G$ by $G_w$ (see also \cite[Proposition 8.1]{Laurent2}), provided that $S$ satisfies (i)--(ii) of Theorem \ref{control_n} and that $F$ satisfies the analogue of (iii). Properties (i)--(ii) for $S$ follow from those for $W$ through the Duhamel formula
$$
S(t,r)u_0=W(t-r)u_0+i\int_r^tW(t-\rho)\big(V_1S(\rho,r)u_0+V_2\overline{S(\rho,r)u_0}\big)\,d\rho
$$
(here $i$ may be taken out of the integral, since $W$ is complex-linear), Gronwall's lemma and the boundedness of $V_1,V_2$ on $H^s_D$; this applies both to $X_T=C([0,T];H^1_0)\cap L^4(0,T;L^\infty)$ ($s=1$) and to $X_T=C([0,T];H^2_D)$ ($s=2$), for $n\in\{2,3\}$. Property (iii) is replaced by
$$
\begin{gathered}
\|F(w,z)\|_{L^1(0,T;H^s_D)}\leq C\big(\|z\|^2_{X_T}+\|z\|^3_{X_T}\big),\\
\|F(w,z)-F(w,\tilde z)\|_{L^1(0,T;H^s_D)}\leq C\big(\|z\|_{X_T}+\|\tilde z\|_{X_T}+\|z\|^2_{X_T}+\|\tilde z\|^2_{X_T}\big)\|z-\tilde z\|_{X_T},
\end{gathered}
$$
which follow as in Propositions \ref{H1-Control} and \ref{H2-Control}, using $w\in C([0,T];W^{1,\infty})$ for the quadratic terms when $s=1$ and the algebra property of $H^2$ when $s=2$. With these estimates, $$\|\mathcal T_w(z)\|_{X_T}\leq C_1\delta'+C_2(\|z\|^2_{X_T}+\|z\|^3_{X_T})$$ and $$\|\mathcal T_w(z)-\mathcal T_w(\tilde z)\|_{X_T}\leq C_2(\|z\|_{X_T}+\|\tilde z\|_{X_T}+\|z\|^2_{X_T}+\|\tilde z\|^2_{X_T})\|z-\tilde z\|_{X_T},$$ where $$\delta':=\|z_0\|_{H^s_D}+\|z_T\|_{H^s_D}\leq2\delta.$$ Hence $\mathcal T_w$ maps the closed ball of radius $M=4C_1\delta$ of $X_T$ into itself and is a contraction there, for $\delta$ small enough.
\end{proof}

\begin{remark}
\textup{(UCP$_w$)} is used only once, in the proof of Proposition \ref{obs_w}, for the weak limit of the sequence used in the contradiction argument, which vanishes in $\omega\times(0,T)$; thanks to Lemma \ref{regularity_w}, it is enough to assume it at the energy level $C([0,T];H^1_0(\Omega))$. For $w\equiv0$ it holds under the non-trapping assumption (Lemma \ref{UCP}); for general $w$ we refer to Section \ref{Sec5}, item (A).
\end{remark}

\section{Concluding remarks and open problems \label{Sec5}}

\subsection{Scope of the results}
The results of this paper show that, near the zero solution, the cubic and quintic nonlinear Schr\"odinger equations can be controlled from the outer boundary of a domain surrounding an obstacle, or from a neighbourhood of infinity in an exterior domain, in $H^1$ and in $H^2$ and in dimensions two and three, under the sole assumption that the obstacle is non-trapping; the linear equation is controllable in the whole scale $H^s_D(\Omega)$, $s\in[-2,2]$. Near a nontrivial trajectory satisfying \textup{(W)}, internal controllability in the exterior domain holds under the additional unique continuation hypothesis \textup{(UCP$_w$)}. Non-trapping is the natural geometric condition here: by Remark \ref{rem_geo} it is equivalent to the geometric control condition for $\omega=\{|x|>R\}$, which is the classical sufficient condition for the observability of the Schr\"odinger equation on bounded domains \cite{Lebeau} (although, unlike for the wave equation, it is not necessary in general; see, e.g., \cite{RoZhaMMM} and \cite{BuZw}); in the present unbounded setting, its sufficiency is precisely what the local smoothing approach provides. In particular, the multiplier condition of Machtyngier \cite{Mach} (Remark \ref{rem_geo}), which in our geometry amounts to the star-shapedness of $\Theta$, is not needed.

From the methodological point of view, the main point is that the local smoothing effect of \cite{BGT1} plays two roles at once, as in the stabilization result of \cite{BoCa}. It replaces the multiplier identity in the proof of observability, and thus removes any boundary term on $\partial\Theta$; and it provides the compactness and the gain of regularity needed in the compactness--uniqueness argument, without proving any new microlocal propagation result up to the boundary (the geometric information is entirely contained in the estimates of \cite{BGT1}). This mechanism only uses that the uncontrolled region $\Omega\setminus\omega$ is bounded and that the obstacle is non-trapping, and it should apply to other dispersive equations in exterior domains for which a local smoothing effect is available.

\subsection{Open problems}
We conclude with some questions left open by this work.

\smallskip
\noindent\textbf{(A) Unique continuation around nontrivial trajectories.} Theorems \ref{lin_ar} and \ref{nlin_ar} rely on \textup{(UCP$_w$)}, as the corresponding results in \cite{Laurent2,CaCoCa}. For $w\equiv0$, this property follows from the spectral argument of Lemma \ref{UCP}. Still, for a general trajectory the potentials $V_1,V_2$ depend on time and are in general not analytic, so that neither spectral nor Holmgren-type arguments apply. Carleman estimates \cite{LaTriZhang,MeOsRo} are the natural tool; they usually require additional geometric assumptions, such as the star-shapedness of $\Theta$, and proving \textup{(UCP$_w$)} for every non-trapping obstacle remains open.

\smallskip
\noindent\textbf{(B) Lower regularity.} The linear control result holds in the whole scale $H^s_D(\Omega)$, $s\in[-2,2]$ (Theorem \ref{lin_th}), and the nonlinear results extend to $s\in(\frac n2,2]$ (Remark \ref{rem_algebra}). Below the energy space, nonlinear results in $H^s_D(\Omega)$, $s\in[0,1)$, would require a Cauchy theory with a forcing term and estimates of the nonlinearity in a suitable space $X_T\subset C([0,T];H^s_D(\Omega))$, as in Theorem \ref{control_n}. In dimension three, the intermediate range $s\in(1,\frac32]$ would require, in addition to the $L^4_tL^\infty_x$ bound, Moser-type estimates in $H^s_D(\Omega)$. The $L^2$ case in dimension two is more delicate, since the cubic equation is then mass-critical and \cite{BGT1} provides no Cauchy theory at this level.

\smallskip
\noindent\textbf{(C) The quintic equation beyond the zero solution.} Corollary \ref{quintic} is a local result near the zero solution. An analogue of Theorems \ref{lin_ar} and \ref{nlin_ar} for the quintic equation, with the potentials $V_1=3\lambda|w|^4$ and $V_2=2\lambda|w|^2w^2$, should follow along the same lines, under the corresponding unique continuation hypothesis. Large-data controllability results in the energy-critical case remain open even on $\mathbb{R}^3$ \cite{BrCaCaDo}; one possible approach would combine a stabilization mechanism with local controllability results, as in item (F) below.

\smallskip
\noindent\textbf{(D) Neumann boundary conditions.} The Neumann problem \eqref{s1}--\eqref{s3} requires the exterior theory for the Neumann Laplacian, for which local smoothing holds with constants depending on the time interval \cite[Remark 2.9]{BGT1}, as well as a finer analysis of the Neumann traces on $\Gamma_1$.

\smallskip
\noindent\textbf{(E) Trapping obstacles.} For the linear equation, boundary controllability from $\partial B_{R_0}$ still holds for some trapping configurations of strictly convex obstacles \cite{Burq93,BuZw}, at the price of a loss of derivatives; see Figure \ref{fig_geo}(c). Extending the nonlinear results to such geometries would require handling this loss in the perturbation argument.

\smallskip
\noindent\textbf{(F) Global results.} In dimension two and in the defocusing case, the exponential stabilization of \cite{CaCoCa}, combined with the local controllability results of this paper, should lead to global controllability in large time, in the spirit of \cite{DeGeLe,Laurent2}. Making this precise requires matching the functional settings and the geometric assumptions of both works.

\subsection*{Acknowledgments:}  This work began in 2014 while the first author was a visiting researcher at the University of Cincinnati, with further progress made in 2023 as a visiting researcher at Virginia Tech. The first author thanks both host institutions for their warm hospitality.

\end{document}